\documentclass[12pt]{article}

\usepackage[a4paper,
textwidth = 16cm,
textheight = 22.5cm,
top = 3cm]{geometry}
\usepackage[colorlinks,citecolor=blue,urlcolor=blue,breaklinks]{hyperref}

\usepackage{latexsym,amsmath,amssymb,amsfonts,amsthm,bm,bbm,soul,graphicx, dsfont,caption, comment, paralist, mathrsfs, enumerate,booktabs,subcaption}
\usepackage[dvipsnames]{xcolor}

\graphicspath{{./figures/}}
\usepackage[normalem]{ulem}
\usepackage[ruled, vlined, linesnumbered]{algorithm2e}
\usepackage{multirow}
\usepackage[scientific-notation=true]{siunitx}
\usepackage[american]{babel}

\usepackage{lmodern}

\newcommand{\bbE}{\mathbb{E}}
\newcommand{\bbP}{\mathbb{P}}
\newcommand{\bbind}{{1}}
\newcommand{\bbR}{\mathbb{R}}
\newcommand{\bbN}{\mathbb{N}}

\newcommand{\calB}{\mathcal{B}}

\newcommand{\calF}{\mathcal{F}}

\newcommand{\scrI}{\mathscr{I}}

\newcommand{\iid}{{IID}\xspace}
\DeclareMathOperator{\adj}{adj}
\DeclareMathOperator{\diag}{diag}

\newcommand{\lrbkt}[1]{\{#1\}}
\newcommand{\asconv}{\xrightarrow{\text{a.s.}}}

\usepackage{xspace}

\newcommand{\pa}{{PA}\xspace}

\newcommand{\mle}{maximum likelihood estimator\xspace}
\newcommand{\qmle}{pseudo maximum likelihood estimator\xspace}
\newcommand{\sumk}{\sum_{k=1}^\infty}
\newcommand{\sumj}{\sum_{j=1}^\infty}
\newcommand{\sss}{\scriptscriptstyle}
\newcommand{\ra}{\rightarrow}

\newcommand{\rkn}{\Hat{r}_k(n)}
\newcommand{\pkzero}{p_k^{\sss (0)}}
\newcommand{\qkzero}{q_k^{\sss (0)}}
\newcommand{\ee}{{empirical estimator}\xspace}

\input{./mis/aad.tex}
\theoremstyle{plain}
\newtheorem{theorem}{Theorem}[section]
\newtheorem{lemma}[theorem]{Lemma}

\newtheorem{corollary}[theorem]{Corollary}
\newtheorem{proposition}[theorem]{Proposition}

\theoremstyle{remark}

\numberwithin{equation}{section}

\newcommand\malt{{\lambda^*}}
\renewcommand\qkzero{{p^{\sss (0)}_{>k}}}

\usepackage[round]{natbib}
\usepackage{comment}

\title{Statistical Inference in Parametric Preferential Attachment Trees}
\author{Fengnan Gao\footnote{Fudan University and Shanghai Center for Mathematical Sciences. Email: \url{fngao@fudan.edu.cn}} \; and\; Aad van der Vaart\footnote{TU Delft. Email: \url{a.w.vandervaart@tudelft.nl}}}
\date{(\today)}
\date{(August 16, 2022)}

\begin{document}
\maketitle

\begin{abstract}
    The preferential attachment (\pa) model is a popular way of modelling dynamic social networks, such as collaboration networks.
    Assuming that the \pa function takes a parametric form, we propose and study the maximum likelihood estimator of the parameter.
    Using a supercritical continuous-time branching process framework, we prove the almost sure consistency
    and asymptotic normality of this estimator.
    We also provide an estimator that only depends on the final snapshot of the network and prove its consistency,
    and its asymptotic normality under general conditions.
    We compare the performance of the estimators to a nonparametric estimator in a small simulation study.
\end{abstract}

\section{Introduction and Notation}
\label{SectionIntroduction}

We study the preferential attachment (\pa) model---a \textit{dynamic} network model which in our setup evolves
from an initial stage consisting of a  single node of degree one (a root node with a loose edge or dead parent)  by
recursively adding at each step a single node and edge.
The incoming node connects to a node in the existing network with probability proportional
to a non-decreasing function of its degree.
The term \textit{preferential attachment} reflects that nodes of higher degrees (``the rich'') inspire more incoming connections (``get richer''),
thus leading to ``the-rich-get-richer'' effect, or the so-called Matthew effect.

For a precise description, define $[n] = \{ 1,\dots, n\}$, and denote the nodes at time $n$ by
$\{ v_i \}_{i\in [n]}$. If the corresponding degrees are  $\{ d_i(n)\}_{i\in [n]}$,
then the node $v_{n+1}$ connects to the existing node $v_i\in\{v_l\}_{l\in [n]}$ with probability proportional
to $f(d_i(n))$, for a given function $f:\bbN_+\to\bbR_+$, i.e., with probability
$$\frac{f(d_i(n))}{ \sum_{j=1}^n f(d_j(n))}.$$
We refer to $f$ as the \textit{preferential attachment} \textit{function},  to $f(k)$ as the \textit{preference} for a node of degree $k$,
and to the denominator in the display as the \textit{total preference} at time $n$.
After the incoming node $v_{n+1}$ has made its choice, the scheme repeats itself at time $n+1$
with the set of existing nodes $\{v_i\}_{i\in [n+1]}$, their updated degrees $\{ d_i(n+1)\}_{i\in [n+1]}$, and the incoming node $v_{n+2}$.
The model may evolve to reach any number of nodes, whence we obtain a random graph that evolves over time.

The \pa model received its modern conception and fame in connection to the prevalence of so-called
\textit{scale-free} networks in the real world, as shown by Albert Barabási and his co-authors (\citet{Barabasi99emergenceScaling, barabasi1999mean,barabasi2000scale}),
and in many subsequent scientific studies in different disciplines.  A simple version of the \pa model
appeared in \citet{Barabasi99emergenceScaling} as a possible explanation of the emergence of the scale-free property
and henceforth the \pa model has remained one of the few \emph{dynamic} models to produce scale-freeness.

Scale-freeness is usually defined in terms of polynomial decay
of the \textit{empirical degree distribution} $P_k(n)$, which is the proportion of nodes of degree $k$ at time $n$:
$$ P_k(n) = \frac{1}{n} \sum_{i \in [n]} \bbind_{\{d_i(n) = k\}}=:\frac 1n N_k(n).$$
In the case that the \pa function $f$ is affine with $f(k) = k + \a$ with $\alpha > -1$, it is known
that $P_k(n) \rightarrow p_k$ almost surely, as $n \rightarrow \infty$, for any fixed $k$, where the limit $(p_k)_{k=1}^\infty$
is a proper probability distribution with atoms proportional to $k^{-(3+\a)}$, as $k\rightarrow\infty$
(up to a slow-varying factor, see \citet{mori2002random,hofstadcomplexnet}).
The Barabási–Albert model is the special case with $f(k) = k$ and gives decrease proportional to $k^{-3}$.

Two different scenarios arise when we no longer restrict attention to the affine \pa functions---superlinear and sublinear.
Roughly speaking, in the superlinear case the \pa function $f$ grows faster than any linear function and the resulting \pa tree looks
like a star with one dominating node (with high probability, the first or the oldest one) connecting to virtually almost every other node, see \citet{oliveira2005connectivity} and references therein.
In the sublinear case $f$ grows more slowly than any linear function, yielding more interesting \pa trees.
The slower growth of strictly sublinear $f$ yields less preference towards high-degree nodes, rendering the emergence of high-degree nodes less likely,
and leading to limiting degree distributions with in general lighter tails than power laws, which have fine and subtle details.


In this paper we consider the statistical estimation of the \pa function from an observed network. We adopt
a parametric specification of a sublinear \pa function and consider statistical inference on the parameter.
A prototype of such a model is  $f(k) = (k + \alpha)^\beta$, for parameters $\a$ and $\beta \le 1$.
This includes the submodel $f(k) = k^{\beta}$, for $\beta\le 1$, considered in \cite[Section 5.8]{barabasi2016network}, and the affine linear
model $f(k)=k+\alpha$. The latter model was considered in \cite{gao2017asymptotic} by relatively direct arguments and martingale methods,
which break down for more general models. The main contribution of the present paper is to exploit the framework
of supercritical Malthusian branching processes due to
\cite{jagers1975branching} and \cite{nerman1981convergence}, and applied to derive the limiting degree
distribution in \pa models by \cite{rudas2007random}, to analyze general parametric models.
The branching process framework allows to study the likelihood function and to prove the consistency and asymptotic normality of the \mle.

We also propose a Wald-type test to test the null hypothesis that the \pa function is affine, thus allowing a test
for scale-freeness versus a well-defined alternative. In the case that only a final snapshot and not the evolution history is observed,
we propose a history-free remedy to the likelihood function and obtain a \qmle, which is shown to be consistent in general and asymptotically normal under general conditions, and those conditions are verified for the case where the \pa function becomes
constant after reaching a certain degree.

An alternative to parametric estimation is the nonparametric approach of \cite{gao2017consistent}, who introduce an empirical estimator
and show that this is consistent for general \pa functions $f$. It is  unknown whether this empirical estimator is asymptotically normal.
We show by simulation that the maximum likelihood estimator of the present paper is significantly more efficient
if the parametric model is correctly specified.

\subsection{Outline of the paper}
The paper is organized as follows.
We present the likelihood and the \mle in Section~\ref{sec:mle}.
Section~\ref{subsec:continuous-tree} introduces supercritical Malthusian branching processes and
deduce the results needed for the present paper (with a more formal introduction in the
appendix, Section~\ref{sec:rooted-tree}).
In Section~\ref{sec:param-consistency}, we prove that the \mle is consistent and in
Section~\ref{sec:param-asymp-normality} that it is asymptotically normal.
To overcome the problem of relying on the entire evolution history of the network, we propose in
Section~\ref{sec:remedy-history} a \qmle, which depends only on the final snapshot of the tree,
and give conditions for its asymptotic normality.
Section~\ref{sec:connect-ee-qmle} gives a new perspective on the \ee, presenting this as a nonparametric version of the \qmle.
In Section~\ref{sec:simu} we present simulation results that demonstrate the performance of the \mle and \qmle,
and compare this to the empirical estimator. 
Section~\ref{sec:proofs} collects all proofs to the results. 
In a second section \ref{sec:janson-truncated} of the appendix we
prove the asymptotic normality of the empirical degree distribution in a new case of special interest,
and give a new proof in the affine case.

\subsection{Notation}
Write $\bbN_+$ for the set of positive natural numbers $\lrbkt{1,2,\dots,}$ and $\bbN$ for the set of natural numbers including zero.
Write $\bbR_+ = (0, \infty)$.
For a sequence $(a_k)_{k=1}^\infty$, define $a_{>k} = \sum_{l >k} a_l$.
Let $N_k(t)$ be the number of nodes of degree $k$ in the network at time $t$,
and $P_k(t) = N_k(t)/t$ the proportion of such nodes.
For a given function $h: \bbN_+\to\bbR$ set  $S_h(t) = \sum_{k=1}^\infty h(k) N_k(t)$.
The superscript $^{\sss (0)}$ stresses that a quantity, such as the limiting degree distribution  $\pkzero$,
is  considered under the true parameter $\theta_0$.
For a vector $v = (v_1, \dots, v_d)^T \in \bbR^d$, let $\| v\| := \| v\|_\infty = \max_{ 1\le i \le d} |v_i|$;
and for a matrix $A \in \bbR^{m\times n}$, set $\| A\| : = \max_{i,j} | A_{ij}|$.
Define the diagonal matrix $\diag(a_1, \dots, a_d)$ by $\bigl(\diag(a_1,\dots,a_d)\bigr)_{ij} = \bbind_{\{i=j\}} a_i$.
Define $u_- = \max(0, -u)$ for $u \in \bbR$.
We write $a \wedge b = \min(a, b) $ for $a, b \in \bbR$.
When $c b_n \le a_n \le C b_n$ for some positive constants $c$ and $C$, we write $a_n \asymp b_n$.

\subsection{Model}
\label{SectionModel}
Throughout the paper $\{f_\theta: \theta\in\Theta\}$ is a collection of  non-decreasing functions $f_\q: \bbN_+ \to \bbR_+$ indexed
by a subset $\Theta\subset\bbR^d$.  It is assumed that every element of the family satisfies one of the two possibilities:
\begin{enumerate}[(i)]
	\item $f_\theta(k) \le C k^{\beta}$ for every $k\in\bbN$ for some positive constants $C$ and $\beta <1$.
	\item $f_\theta (k)= k + \alpha$ for every $k \in \bbN$, for some $\a>-1$.
\end{enumerate}%
It is also assumed that $\q\mapsto f_\q(k)$ is twice continuously differentiable with derivatives denoted by
$\dot f_\q(k)$ and $\ddot f_\q(k)$, which are a vector (gradient) and a matrix (Hessian)  in the case of a multidimensional parameter.
Differentiation with respect to the parameter $\theta$ is also denoted by a dot in general.

\section[Construction of the MLE]{Construction of the Maximum Likelihood Estimator}
\label{sec:mle}
Let $N_k(t)$ be the number of nodes of degree $k\in\bbN$ in the graph with nodes $v_1,\ldots, v_t$.
Initially there is a single node $v_1$ with degree one and hence $N_1(1) = 1$,  and we can set $N_k(1) = 0$, for  every $k \ge 2$,
to define a complete degree sequence. If $D_t$ denotes the degree of the node to which the node
$v_{t}$ is attached, then, for $t\ge 2$,
\begin{equation}
	\begin{aligned}
		N_k(t) = N_k(t-1) + \bbind_{\lrbkt{D_t = k -1}} - \bbind_{\lrbkt{D_t = k}} + \bbind_{\lrbkt{k=1}}.
	\end{aligned}
	\label{eqn:degree-evolution}
\end{equation}
The random graph evolves as a Markov process, and hence the likelihood factorizes as the conditional
likelihoods of the new node given the current tree. Any of the existing nodes may be chosen to attach the new node, but only the
degree $D_t$ of this node is important for the value of the likelihood, which takes the form
\begin{equation*}
	L_n(f_\theta) = \prod_{t=2}^{n} \frac{f_\theta(D_t)N_{D_t}(t-1)}{S_{f_\theta}(t -1 )},
\end{equation*}
where the norming ``constant'' $S_{f_\theta}(t-1)=\sumk f_\q(k) N_k(t-1)$ is the total preference
in the graph with nodes $v_1,\ldots, v_{t-1}$ given the \pa function $f_\theta$.
The total preference can be computed recursively
by the rule $S_f(t) = S_f(t-1) + f(D_{t} +1) - f(D_{t}) + f(1) $, for $t \ge 2$, with the initialization $S_f(1) =  f(1)$.
In particular, the total preference at stage $t$ can be expressed in the degree sequence up to time $t-1$,
and the full likelihood depends on the data only through $D^{(n)}= (D_t)_{t=2}^{n}$.
The normalized log-likelihood up to the term $\sum_{t=2}^n\log N_{D_t}(t-1)$ is given by
\begin{equation}
	\label{eqn:iota-nn}
	\begin{aligned}
		\iota_n(f_\theta)
		 & = \frac{1}{n} \sum_{k=1}^\infty \log f_\theta (k) \sum_{t=2}^{n} \bbind_{\lrbkt{D_t = k}} - \frac{1}{n} \sum_{t=2}^{n} \log S_{f_\theta}(t-1) \\
		 & = \sum_{k=1}^\infty \log f_\theta(k) P_{>k}(n) - \frac{1}{n} \sum_{t=2}^{n} \log S_{f_\theta}(t-1).
	\end{aligned}
\end{equation}
In the last step we use the identity $\sum_{t=2}^{n} \bbind_{\lrbkt{D_t = k}}=N_{>k}(n)$, which is essentially \citep[Lemma 1]{gao2017consistent} and results from the fact that
any node of degree strictly large $k$ in the tree at time $n$ must have been chosen  for attachment
while it had degree $k$, exactly once up until this time (namely when its degree went up from $k$ to $k+1$).

The derivative of the $\log$-likelihood is
\begin{equation}
	\label{eqn:log-likelihood}
	\begin{aligned}
		\dot{\ell}_n(f_\theta)
		= \frac{\partial }{ \partial \theta} \log L_n(f_\theta)
		 & = \sum_{t=2}^{n} \left[ \frac{\dot{f}_\theta}{ f_\theta} (D_t) - \frac{S_{\dot{f}_\theta}(t-1)}{ S_{f_\theta}(t-1)}\right]                                         \\
		 & = \sum_{t=2}^{n}\left[ \frac{\dot{f}_{\theta}}{ f_{\theta} }(D_t) - \bbE_{{\theta}}\Bigl[ \frac{\dot{f}_{\theta}}{f_{\theta}}(D_t) \Big| \calF_{t-1}\Bigr]\right],
	\end{aligned}
\end{equation}
where $(\calF_{t})_{t\ge 1}$ is the filtration generated by the stochastic process of the graph's evolution.  The last expression follows
from the fact that $\bbP(D_t = k | \calF_{t-1})=f_\theta(k)N_k(t-1)/S_{f_\theta}(t-1)$, and shows that the score process is a martingale
under the true parameter $\theta$, as usual, which can also be seen by readily verifying
\(
\mathbb{E}[\dot{\ell}_n(f_\theta) \mid \mathcal{F}_{n-1}] = \dot{\ell}_{n-1}(f_\theta).
\)
Dividing this martingale by the number $n$ of nodes in the network and rewriting as in \eqref{eqn:iota-nn} gives
\begin{equation}
	\begin{aligned}
		\dot{\iota}_n(f_{\theta}) =\frac1n \dot{\ell}_n(f_\theta) & = \sum_{k=1}^\infty \frac{\dot{f}_\theta}{f_\theta}(k) P_{>k}(n) - \frac{1}{n} \sum_{t=2}^{n} \frac{ S_{\dot{f}_\theta}(t-1)}{S_{f_\theta}(t-1)}                                               \\
		                                                          & =  \sum_{k=1}^\infty \frac{\dot{f}_\theta}{f_\theta}(k) P_{>k}(n) - \frac{1}{n} \sum_{t=2}^{n} \frac{\sum_{k=1}^\infty P_k(t-1) \dot{f}_\theta (k)}{ \sum_{k=1}^\infty P_k(t-1) f_\theta (k)}.
	\end{aligned}
	\label{eqn:iota-n}
\end{equation}
The maximum likelihood estimator $\hat \theta_n$ can be defined to be either the maximizer of the log-likelihood \eqref{eqn:log-likelihood} or a solution to the equation $\dot{\iota}_n(f_{\theta})=0$.

To understand the behavior of the \mle, we need to study the rescaled log-likelihood in \eqref{eqn:iota-nn} or
its derivative     \eqref{eqn:iota-n}. However, the quantities in these equations are anything but easy---they are functionals
of the entire evolution history of a complex Markov process, where the influence of the past persists in the likelihood.
Simple and straightforward approaches such as martingale methods (cf.\ \cite[Chapter 8]{hofstadcomplexnet})
are no longer meaningful.
Instead, we employ the theory of supercritical Malthusian branching processes, which will be introduced in the section~\ref{subsec:continuous-tree}.

Adapting such a powerful framework, we will be able to assert the utility of the \mle by showing its consistency and asymptotic normality, respectively.
In particular, in both Theorems~\ref{thm-consistency-param} (from the viewpoint of M-estimator) and \ref{thm-consistency-param2} (from the viewpoint of Z-estimator), we show under some mild assumptions on the parametric family that $\hat{\theta} \to \theta_0$ in an almost sure sense as the number of nodes $n\to \infty$.
Theorem~\ref{thm-param-clt} illuminates that $\sqrt{n} (\hat{\theta} - \theta_0) \rightsquigarrow N(0, \sigma_{f_{\theta_0}}^2)$  for some properly defined $\sigma_{f_{\theta_0}}^2$ such that constructing confidence sets and testing particular well-specified hypotheses are possible.

\section{The continuous random tree model}
\label{subsec:continuous-tree}

Before stating the main results formally, it is necessary to adopt a continuous-time framework, where nodes are added after exponentially distributed
waiting times. To set this up we equip every node $v$
with a pure birth process $\xi_v$ whose events correspond to new, future nodes being attached to this particular node and which
in calendar time starts at its own birth, i.e.\ when it is added to the tree. These birth processes are i.i.d.\ across nodes and
a typical birth process $\xi=(\xi(t): t\ge 0)$ has birth rate equal to $f\bigl(\xi(t)+1\bigr)$. Thus,
$\xi$ is a continuous-time Markov process with state space $\bbN_+$, initial value $\xi(0)=0$, and with
the only possible transitions stepwise increases $k-1\rightarrow k$, determined by
\begin{equation}
	\bbP\bigl(\xi(t+dt) = k \mid \xi(t) = k-1\bigr) = f(k)\, dt + o(dt).
	\label{eqn-birth-process}
\end{equation}
Every birth corresponds to a new node attached to the existing tree at the node whose birth process produced the event.
A node whose birth process has had $k-1$ births will have $k-1$ children and one parent in the tree and hence possess degree $k$.
It will produce a new child with rate $f(k)$, explaining the right side of the display. At a given calendar time $t$ every node in the
current tree $\Upsilon_t$ will have a corresponding, active birth process. The total rate of all active birth
processes will be $S(t)=\sum_{v\in\Upsilon_t} f(d_v(t))$, where $d_v(t)$ is the degree of
$v\in\Upsilon_t$, and a new node will be attached to node $v\in\Upsilon_t$
with probability $f(d_v(t))/S(t)$ after an exponential waiting time with mean $1/S(t)$.
The process $\Upsilon_t$ starts at time $0$ with a single node that is understood to have degree 1 and hence the first
birth will be after an exponential time with mean $f(1)$.
For a more formal setup in the language of general branching processes, see Section~\ref{sec:rooted-tree}.

Thus, we obtain a continuous-time branching process $\Upsilon_t$ that contains the discrete-time process as a skeleton.
We define $T_t$ as the total number of births in the continuous process up until time $t$,
and $\tau_n = \inf\{ t>0: T_t\ge n-1 \}$, for $n=1,2,\dots$, as the time of the $n$th birth (where $\tau_1 = 0$).
When evaluated at the stopping times $\t_1,\t_2,\ldots$, the continuous-time process gives a sequence of trees
$\Upsilon_{\tau_1}, \Upsilon_{\tau_2},\ldots$, that is  equivalent to the \pa model.

The advantage of the continuous-time setup is that the results on branching processes become more straightforward.
To every node we may attach besides a birth process a second continuous-time process, called a \textit{characteristic},
also starting at the birth of the node.  
Just as the birth processes, the characteristics are assumed identically distributed, and the birth process and characteristic attached to a single node may be dependent, but every time a new node is added, a pair of a birth process and a characteristic are created that evolve independently of the processes attached to the  nodes that appeared earlier in the history of the tree. 
For a given characteristic $\varphi$ we consider the process $(Z_t^\varphi: t\ge 0)$ given by
\begin{equation*}
	Z_t^{\varphi} = \sum_{v \in\Upsilon_t} \varphi_v(t- \sigma_v).
\end{equation*}
Here $\sigma_v$ is the calendar time at which node $v$ is added to the tree, so that
$t-\sigma_v$ is the lifetime of the node since its birth. The characteristic of node $v$, denoted by $\varphi_v$, has $\varphi_v(t)$ interpreted as its value at
age $t$ and hence $\varphi_v(t-\sigma_v)$ as its value at calendar time $t$. The variable $Z_t^\varphi$
gives the sum of the characteristics of all individuals in the tree at time $t$. In the supercritical case
the processes $Z_t^\varphi$ grow exponentially in time at a rate $\mathrm{e}^{\malt t}$, where $\malt$ is
the so-called Malthusian parameter, and $\mathrm{e}^{-\malt t} Z_t^\varphi$ tends to a (random) limit as $t\ra\infty$. We shall employ these
limit theorems with appropriate choices of characteristics to derive the asymptotics of the likelihood function.

A key element is the Laplace transform of the \textit{reproduction function} $\mu(t)=\bbE[\xi(t)]$,
the mean number of births of a single node at age $t$, which in our case can be expressed in the \pa function as
(see \cite{rudas2007random} or the proof in Section~\ref{sec:proofs})
\begin{equation}
	\label{laplace-transform}
	\rho_f(\lambda):= \int_0^\infty \mathrm{e}^{-\lambda t}\,\mu(dt)
	=\sum_{l=1}^\infty \prod_{k=1}^l \frac{f(k)}{\lambda+f(k)}.
\end{equation}
The function $\rho_f$  is convex and decreasing on its domain (the set where it is finite), which is an interval $(\underline\lambda, \infty)$
or $[\underline\lambda, \infty)$ in the positive half line, and tends to zero as $\lambda\ra\infty$.
The Malthusian parameter $\malt$ is the solution of the equation $\rho_f(\malt)=1$.

In the case of a strictly sublinear \pa function $f$, we have $\underline\l=0$ and $\rho_f(\lambda)\uparrow \infty$
as $\lambda\downarrow0$, while for the \pa function $f(k)=k+\alpha$ we have $\underline\lambda=1$ and
the exact form $\rho_f(\lambda)= (1+\alpha)/(\lambda-1)$ is known (see \cite{gao2017consistent,rudas2007random} or the proof of Proposition~\ref{lemma-conv-phi1-phi2}).
In both cases the range of $\rho_f$ contains the point $1$ as an interior point and the Malthusian parameter exists.
If the sublinearity assumption is violated in the sense that neither of the two sublinear conditions holds, the existence of the Malthusian parameter is not guaranteed, and the associated branching process could explode or behave irregularly, whence the branching process framework and the ensuing may results fail. 

Furthermore, under the sublinear assumptions of the \pa function, we have the following limit theorem.

\begin{proposition}
	\label{lemma-conv-phi1-phi2}
	Suppose that the  range of $\rho_f$ contains an open neighborhood of 1 and
	Let $\varphi_1$ and $\varphi_2$ be monotone increasing characteristics such that, for some constants $C>0$ and $\gamma\ge 0$
	and every $t>0$,
	\begin{equation}
		\label{eqn:max-phi1-phi2}
		\varphi_i(t) \le C\, \xi(t)^2,\qquad i=1,2, \quad \text{almost surely}.
	\end{equation}
	If $f: \bbN_+\to \bbR_+$ is monotone with $f(k)\le C k^\b$ for some constants $C$ and $\b<1$, then, as $t \rightarrow \infty$,
	\begin{equation}
		\frac{Z_t^{\varphi_1}}{Z_t^{\varphi_2}} \asconv 
		\frac{ \int_0^\infty \mathrm{e}^{- \malt t} \bbE [ \varphi_1(t) ]\, dt }{\int_0^\infty \mathrm{e}^{- \malt t} \bbE [ \varphi_2(t) ]\, dt} .
		\label{eqn-conv-phi1-phi2}
	\end{equation}
	The same is true if $f(k)=k+\alpha$, for some $\a>-1$, provided that for some $r<2+\alpha$ with $r\le 2$,
	\begin{equation}
		\label{eqn:max-phi1-phi2-affine}
		\varphi_i(t) \le C\,  \xi(t)^r,\qquad i=1,2, \quad \text{almost surely}.
	\end{equation}
\end{proposition}

For a given \pa function $f$ with Malthusian parameter $\malt$, define
\begin{equation}
	p_k = \frac{\malt}{\malt + f(k)}\prod_{j=1}^{k-1} \frac{f(j)}{\malt + f(j)}, \qquad  k\in \bbN_+,
	\label{eqn-pk-limit}
\end{equation}
where the empty product is defined to be $1$, so that $p_1 = \malt/\bigl(\malt + f(1)\bigr)$.
The definition of $\malt$ may be used to show that $(p_k)_{k=1}^\infty$  is a probability distribution on $\bbN_+$.
In fact, it is the limit of the empirical degree distribution $(P_k(n))_{k=1}^\infty$ of the \pa network, as shown
by \cite{rudas2007random}. More generally, we have the following limiting result.

\begin{corollary}
	\label{lemma-hk-pk-conv}
	If  $f: \bbN_+\to \bbR_+$ is monotone increasing and satisfies $f(k)\le C k^\beta$, for some $\beta<1$, and
	$h: \bbN_+ \rightarrow \bbR_+$ satisfies  $h(k)\le Ck^2$, for a constant $C$, and every $k$, then the
	empirical degrees $P_k(n)$ in the model with \pa function $f$ satisfy, as $n \rightarrow \infty$,
	\begin{equation}
		\sumk h(k) P_k(n) \asconv \sumk h(k) p_k.
		\label{eqn-hk-pk-conv}
	\end{equation}
	The same is true for the \pa function given by  $f(k)=k+\a$, for some $\a>-1$, and every
	$h: \bbN_+ \rightarrow \bbR_+$ satisfying $h(k)\le C k^r$, for some $r<2+\alpha$ with $r\le 2$.
\end{corollary}

Choosing $h$ equal to the indicator of the set $\{k\}$ for a given $k$, we recover the convergence
$P_k(n)\rightarrow p_k$ of the empirical degrees to the limit $p_k$, first obtained in
\cite{rudas2007random}.

It is worth noting that the tail of the distribution $(p_k)_{k=1}^\infty$ (of which the dependence on the \pa function $f$ is
suppressed from the notation), as $k \rightarrow \infty$ is  heaviest among sublinear $f$ when $f$ is affine with
$f : k \mapsto k + \alpha$, and it corresponds to the limiting (asymptotic) power law with exponent  $3+\alpha$.

The following lemma records two useful identities, which readily follow from the definition \eqref{eqn-pk-limit}
of $p_k$, and the definition of $\malt$ in terms of the Laplace transform \eqref{laplace-transform}.

\begin{lemma}
	\label{lemma-limit-degree-dist-equality}
	Suppose  $(p_k)_{k=1}^\infty$ is the limiting degree distribution specified in \eqref{eqn-pk-limit} for the \pa function $f$ with
	Malthusian parameter $\malt$.  Then $\malt=\sum_{j=1}^\infty f(j)p_j$ and, for all $k\geq 1$,
	\begin{align}
		\label{eqn-lemma-limit-degree-dist-equality}
		p_{>k} & =\frac{f(k)p_k}{\sum_{j=1}^\infty f(j)p_j}.
	\end{align}
\end{lemma}

\section{Consistency}
\label{sec:param-consistency}
The following theorem shows that the \mle in the model introduced in Section~\ref{SectionModel} is consistent.
We assume that the true parameter $\q_0\in\Theta$ is \textit{identifiable} in the model $\{ f_\theta : \theta \in \Theta \}$
in the sense that $f_\q(k)= c f_{\q_0}(k)$ for every $k\in\bbN_+$ and some constant $c$ if and only if $\q=\q_0$.

\begin{theorem}
	\label{thm-consistency-param}
	In the model $\{ f_\theta : \theta \in \Theta \}$ stated in Section~\ref{SectionModel} with compact parameter
	space $\Theta\subset \bbR^d $ and \pa functions satisfying $f_\q(k)\le C k^\b$ for some  constants  $C$ and $\beta<1$ or
	$f_\q(k)=k+\a$ for some constant $\a>-1$, for every $k$ and  every $\q\in\Theta$,
	the \mle $\hat\theta_n$ satisfies $\hat{\theta}_n \rightarrow \theta_0$  almost surely  under $\theta_0$.
\end{theorem}

\subsection{Identifiability From Score Equation}
\label{subsec:iden-score}
For computational ease the \mle may be characterized as a solution to the likelihood
equations $\dot{\iota}_n(f_\theta)=0$, for $\dot{\iota}_n$ given in \eqref{eqn:iota-n}. The asymptotic version
of this function is
\begin{equation}
	\dot{\iota}(f_\theta) = \sum_{k=1}^\infty \frac{\dot{f}_\theta }{ f_\theta}(k) p_{>k}^{\sss (0)}
	- \frac{\sum_{k=1}^\infty p_k^{\sss (0)} \dot{f}_\theta (k)}{ \sum_{k=1}^\infty p_k^{\sss (0)} f_\theta (k)}.
	\label{eqn:iota-dot}
\end{equation}
It follows from \eqref{eqn-lemma-limit-degree-dist-equality} that the true parameter
$\q_0$ solves the equation $\dot{\iota}(f_\theta)=0$. The following proposition shows that the processes
$\dot{\iota}_n$ tend uniformly to $\dot\iota$. The proof is similar to the proof of
Theorem~\ref{thm-consistency-param} and will be omitted.

\begin{proposition}
	\label{prop-uniform-conv-likelihood}
	In the model $\{ f_\theta : \theta \in \Theta \}$ stated in Section~\ref{SectionModel} with compact parameter
	space $\Theta\subset \bbR^d $ and \pa functions satisfying $f_\q(k)\le C k^\b$ for some  constants  $C$ and $\beta<1$ or
	$f_\q(k)=k+\a$ for some constant $\a>-1$, for every $k$ and  every $\q\in\Theta$, assume that, for some constants $C$ and $\g$,
	\begin{align*}
		\bigl\|\dot f_\theta(k)\bigr\|          & \le C k \log^\gamma k, \\
		\Bigl\|\frac{\dot f_\q}{f_\q}(k)\Bigr\| & \le C \log^\gamma k.
	\end{align*}
	Then $\sup_{\theta \in \Theta} |\dot{\iota}_n(f_\theta) - \dot{\iota}(f_\theta) | \rightarrow 0$, almost surely,
	as $n \rightarrow \infty$.
	\label{eqn-uniform-conv}
\end{proposition}

It follows that the maximum likelihood estimator is asymptotically the \textit{unique} solution to the
likelihood equations in compact subsets of the parameter space  in which  $\q_0$ is the unique
zero of $\q\mapsto \dot{\iota}(f_\q)$ (compare Theorem~5.9 in \cite{vdvaart2000asymptotic}).
However, in general proving global uniqueness turns out to be  {difficult}. We present
the following partial results, starting with two useful lemmas.

\begin{lemma}
	Suppose $(v_k)_{k=1}^\infty$ is strictly decreasing with respect to $k$.
	If  $(p_k)_{k=1}^\infty$ and $(q_k)_{k=1}^\infty$ are probability distributions on $\bbN_+$ such that
	\( p_k \le q_k  \) for $k \le K$ and \( p_k > q_k \) for $k > K$, then
	\begin{equation*}
		\sumk p_k v_k > \sum_{k=1}^\infty q_k v_k.
	\end{equation*}
	In case $(v_k)_{k=1}^\infty$ is strictly increasing, the inequality is true in the opposite direction.
	\label{lemma-two-prob-vk}
\end{lemma}

\begin{lemma}
	For a probability distribution $(p_k)_{k=1}^\infty$ and nonnegative sequence $(w_k)_{k=1}^\infty$
	such that $\sum_{k=1}^\infty p_k w_k < \infty$, define $(q_k)_{k=1}^\infty$ by $q_k = p_k w_k / \sum_j p_j w_j$.
	If $(w_k)_{k=1}^\infty$ is strictly increasing, then there exists a $K$ such that $p_k \ge q_k $ for $k \le K$ and $p_k < q_k$ for $k > K$.
	If $(w_k)_{k=1}^\infty$ is strictly decreasing, then there exists a $K$ such that $p_k \le q_k $ for $k \le K$ and $p_k > q_k$ for $k > K$.
	\label{lemma-reweighted-prob}
\end{lemma}

We say that changing  from $\theta_0$ to $\theta$
\textit{induces monotonicity} if $f_\theta(k)/f_{\theta_0}(k)$ is either strictly increasing or strictly decreasing in $k$.

\begin{lemma}
	If changing from $\q_0$ to $\q$ induces monotonicity on $f_\theta/f_{\theta_0}$
	for every $\theta$ in a subset $\Theta' \subset \Theta$, then  $\dot{\iota}(f_\theta) \neq 0$, for every
	$\theta \in \Theta'$.
	\label{lem:monotone-f-neq}
\end{lemma}

For illustrative purposes, we next study in detail different variants of the parametric form $k \mapsto (k+\alpha)^\beta$ in the sections \ref{sec:ex-alpha-beta}--\ref{sec:ex-alpha}.
However, we point out our results work in much more general capacity, as studied in section \ref{sec:ex-general}.  For instance, one could easily work out that our results apply to the parametric specification \(k \mapsto (\log(k+\alpha))^\beta\). 
We provide a toy example in Section~\ref{sec:ex-log-beta} for such a parametric form besides $k \mapsto (k+\alpha)^\beta$ and their variants.

\subsubsection{The model $f_{\alpha, \beta}(k) = (k+\alpha)^\beta$}
\label{sec:ex-alpha-beta}
In the case that $f_{\alpha, \beta}(k) = (k+\alpha)^\beta$,  changing the parameter $(\alpha_0, \beta_0)$ to
another parameter $(\a,\b)$ does not always induce monotonicity on $f_{\alpha,\beta}/ f_{\alpha_0,\beta_0}$.
An analysis of the derivative of the function $g(x) = (x + \alpha)^\beta / (x+\alpha_0)^{\beta_0}$  yields that
$f_{\alpha,\beta}/f_{\alpha_0,\beta_0}(k)$ is  increasing in $k$ on
$\{ (\alpha,\beta): \beta - \beta_0 + \beta \alpha_0 - \beta_0 \alpha \ge 0, \beta \ge \beta_0 \}$
and is decreasing on the set $\{ (\alpha,\beta) : \beta - \beta_0 + \beta \alpha_0 -\beta_0\alpha \ge 0, \beta \le \beta_0 \}$.
The preceding technique does prove that $(\alpha_0, \beta_0)$ is a unique root of $\dot{\iota}(f_{\alpha,\beta})=0$
in these sets, but this does not exhaust the full parameter set.

However, the lemmas ~\ref{lemma-reweighted-prob} and \ref{lemma-two-prob-vk} may be used to prove local uniqueness. The Hessian matrix of
$\iota(f_{\alpha,\beta})$ evaluated at $(\alpha_0,\beta_0)$ can be calculated as, with
the shorthand notation $f_0 := f_{\alpha_0,\beta_0}$,
\begin{equation}
	\ddot{\iota} (f_{\alpha_0, \beta_0}) = - \frac{1}{a^2}
	\begin{pmatrix}
		ab - c^2 & ad-ce \\ ad-ce\ \  & af - e^2
	\end{pmatrix},
	\label{eqn-hessian-iota-alpha-beta}
\end{equation}
where the quantities $a, b, c,d, e, f$ are defined as
\begin{alignat*}{4}
	 & a &  & = \sumk \pkzero f_0 (k),                              &  & b &  & = \sumk \pkzero f_0(k) \frac{\beta_0^2}{(k+ \alpha_0)^2},                 \\
	 & c &  & = \sumk \pkzero f_0(k) \frac{\beta_0}{ k + \alpha_0}, &  & d &  & =  \sumk \pkzero f_0(k) \frac{\beta_0}{k + \alpha_0} \log (k + \alpha_0), \\
	 & e &  & = \sumk \pkzero f_0(k) \log (k + \alpha_0),\qquad     &  & f &  & =\sumk \pkzero f_0(k) \log^2 (k + \alpha_0).
\end{alignat*}
It follows that $ab- c^2$ is strictly positive, since by the Cauchy--Schwarz inequality,
where $K$ follows the law $(\pkzero)_{k=1}^\infty$,
\begin{equation*}
	\Bigl( \bbE_{p_0}\Bigl[f_0(K)\frac{\beta_0}{ K+ \alpha_0} \Bigr] \Bigr)^2
	< {\bbE_{p_0} [f_0 (K)]} \, {\bbE_{p_0}\Bigl[ f_0(K) \frac{\beta_0^2}{(K+\alpha_0)^2}\Bigr]}.
\end{equation*}
The same arguments work to prove that $af - e^2 >0$ and $bf - d^2 >0$.  The determinant of the Hessian matrix is given by
\begin{align*}
	|\ddot{\iota}(f_{\alpha_0,\beta_0})| & = a^2 bf + c^2 e^2 - abe^2 - afc^2 - a^2 d^2 - c^2 e^2 + 2 acde \\
	                                     & = a^2(bf-d^2) - ae(be-cd) - ac(cf-ed).
\end{align*}
This can be shown to be strictly positive by showing that both $be -cd <0$  and $ cf - ed <0$.
We shall  prove $be <cd$; the proof that also $cf <ed$ is similar.
Define $x_k = \pkzero f_0(k) \beta_0/(k+\alpha_0)$ and $u_k = (k+\alpha_0) \log(k+\alpha_0) / \beta_0$, $y_k = x_k u_k = \pkzero f_0(k) \log(k+\alpha_0)$.  Define $p_x(k) = x_k / \sum_{j} x_j$ and $p_y(k) = y_k / \sum_{j} y_j$.
Since $u_k$ is strictly monotone decreasing, an application of Lemma~\ref{lemma-reweighted-prob}
tells us that there exists a $K$ such that $p_x(k) \ge p_y(k)$ for $k \le K$ and $p_x(k) < p_y(k)$ for $k > K$.
Applying Lemma~\ref{lemma-two-prob-vk} with $w_k = \beta_0/(k+\alpha_0)$, we see that $be < cd$.

We conclude that the Hessian matrix is negative definite, so that $(\alpha_0, \beta_0)$ is
a unique root of $\dot{\iota}(f_{\alpha,\beta})$ in a neighborhood around $(\alpha_0,\beta_0)$.

\subsubsection{Global concavity  of the model $f_\beta (k) = (k + \alpha_0)^\beta$ with known {$\alpha_0$}}
\label{sec:ex-beta}
The single-parameter model $f(k)=k^\beta$ is a main example of sublinear \pa, treated in \cite[Section 5.8]{barabasi2016network}.
We allow in addition a nonzero offset $\a_0$ and consider $f_\beta(k) = (k + \alpha_0)^\beta$, with $\beta$ the only parameter.
The limit function \eqref{eqn:iota-dot} reduces to
\begin{equation*}
	\dot\iota(f_\beta)   = \sum_{k=1}^\infty \log (k + \alpha_0) p_{>k}^{\sss (0)}  -\frac{\sum_{k=1}^\infty p_k^{\sss (0)} (k + \alpha_0)^\beta \log (k+\alpha_0) }{ \sum_{k=1}^\infty p_k^{\sss (0)} (k+\alpha_0)^\beta} .
\end{equation*}
The second order derivative can be computed as
\begin{align*}
	\ddot\iota(f_\beta)
	= -\frac{\sum_{k=1}^\infty p_k^{\sss (0)} (k + \alpha_0)^\beta \log^2 (k+\alpha_0)}{\sum_{k=1}^\infty p_k^{\sss (0)} (k + \alpha_0)^\beta}
	+ \Bigl(\frac{ \sum_{k=1}^\infty p_k^{\sss (0)} (k + \alpha_0)^\beta \log (k + \alpha_0)  }{\sum_{k=1}^\infty p_k^{\sss (0)} (k+\alpha_0)^\beta}\Bigr)^2.
\end{align*}
This can be seen to be strictly negative for any $\beta \in [0,1]$, as a consequence of the Cauchy--Schwarz inequality,
with $K$ following the law $(p_k^{\sss (0)})_{k=1}^\infty$,  
\begin{align*}
	\bigl( \bbE_{p_0} \bigl[(K + \alpha_0)^{\beta}\log (K + \alpha_0)\bigr]\bigr)^2
	  < \bbE_{p_0}\bigl[ (K + \alpha_0)^{\beta}\log^2 (K+\alpha_0)\bigr] \bbE_{p_0} (K + \alpha_0)^\beta.
\end{align*}
Thus, the limiting log likelihood is concave and the root of the limiting score function is unique.
Another perspective is that moving the parameter $\beta$ away from $\beta_0$ induces monotonicity on
$f_\beta/f_{\beta_0}(k) = (k+\alpha_0)^{\beta-\beta_0}$.

\subsubsection{Almost-global uniqueness in case of $f(\alpha) = (k + \alpha)^{\beta_0}$ with known {$\beta_0$}}
\label{sec:ex-alpha}
For any $\alpha \neq \alpha_0$,  the function $k\mapsto (k+\alpha)^{\beta_0}/(k+\alpha_0)^{\beta_0}$
is monotone increasing (when $\alpha < \alpha_0$) or decreasing (when $\alpha > \alpha_0$).
Applying Lemma~\ref{lem:monotone-f-neq}, we conclude that the root is unique in every bounded domain.

\subsubsection{Global concavity of the model $f(\beta) = \log^\beta (k)$}
\label{sec:ex-log-beta}
For further illustrative purposes, we study the toy case of the \pa function being $k \mapsto \log^\beta(k)$, where $\beta$ is the only parameter.
We easily calculate the score function and its derivative as follows 
\begin{equation*}
    \begin{aligned}
        \dot{\iota}(f_\beta) & = \sum_{k=1}^\infty \log \log (k) \qkzero - \frac{\sum_{k=1}^\infty \pkzero \log^\beta (k) \log \log (k)}{ \sumk \pkzero \log^\beta (k) },\\
        \ddot{\iota}(f_\beta) & = 
        -\frac{\sum_{k=1}^\infty \pkzero \log^\beta (k) ( \log \log (k) )^2 }{\sum_{k=1}^\infty \pkzero \log^\beta (k)}
        + \biggl(\frac{ \sum_{k=1}^\infty \pkzero  \log^\beta (k) \log \log (k)  }{\sum_{k=1}^\infty \pkzero \log^\beta (k)}\biggr)^2.
    \end{aligned}
\end{equation*}
Similar to the analysis in Section~\ref{sec:ex-beta}, $\ddot{\iota}(f_\beta)$ here is strictly negative for any $\beta>0$ by the Cauchy--Schwarz inequality
\begin{align*}
	\bigl( \bbE_{p_0} \bigl[\log^\beta (K) \log \log (K) \bigr]\bigr)^2
	  < \bbE_{p_0}\bigl[ \log^\beta(K) (\log \log (K))^2 \bigr] \bbE_{p_0} [\log^\beta(K)],
\end{align*}
where $K \sim (\pkzero)_{k=1}^\infty$ is an auxiliary random variable.
As such, the limiting score function is monotone decreasing with respect to the parameter $\beta$, and has a unique zero at \( \beta = \beta_0 \).
Furthermore, the limiting log-likelihood is concave with a unique maximizer.  

\subsubsection{The case of general {$f_\theta$}}
\label{sec:ex-general}
In practice, it could happen that the score function \eqref{eqn:iota-n} has multiple roots,
particularly if its limit \eqref{eqn:iota-dot} has multiple roots.
In such cases, we may employ the \textit{empirical estimators} in \cite{gao2017consistent} to identify the correct one.
These are defined as follows (see their Equation (2)):
\begin{equation*}
	\hat{r}_k(n) = \frac{N_{>k}(n)}{N_k(n)}.
\end{equation*}
In  \cite{gao2017consistent} these estimators are shown to converge to $f_{\theta_0}(k)/(\sumj f_{\theta_0}(j)p^{\sss (0)}_j)$, almost surely.
This suggests an \textit{empirical estimator} for the parameter $\q \in \Theta \subset \bbR^{d}$ as the solution of  the system of equations:
\begin{equation}
	\label{eqn:ee-equation}
	\frac{f_\theta(k)}{f_\theta(1)} = \frac{\hat{r}_{k}(n)}{ \hat{r}_1(n)},\qquad  k=2,3,\ldots,d+1.
\end{equation}
Two computational strategies suggest themselves.  If finding the set of solutions
to the likelihood equations $\dot{\iota}_n(f_{\theta}) = 0$ is easier than solving \eqref{eqn:ee-equation},
then we may select from this set the solution that minimizes
\begin{equation*}
	\sum_{k=2}^{d+1} \left| \frac{f_\theta(k)}{f_\theta(1)} - \frac{\hat{r}_{k}(n)}{ \hat{r}_1(n)} \right|.
\end{equation*}
On the other hand, if \eqref{eqn:ee-equation} is easier to solve than \eqref{eqn:iota-n}, then we
may find the solution of likelihood equations in a neighborhood of the solution of \eqref{eqn:ee-equation},
possibly by an iterative scheme such as Newton's algorithm.

In both cases the resulting estimator will be consistent.

\begin{theorem}
	Under the conditions of Theorem~\ref{thm-consistency-param}, the
	solution of the likelihood equation resulting from either of the two indicated procedures  is almost surely
	consistent for $\q_0$.
	\label{thm-consistency-param2}
\end{theorem}

\section{Asymptotic Normality}
\label{sec:param-asymp-normality}
We prove asymptotic normality of consistent solutions to the likelihood equations, as follows. Recall that for $\theta \in \Theta \subset \bbR^d$ and every $k \in \bbN_+$, $\dot{f}_\theta(k)\in\bbR^d$ and $\ddot{f}_\theta(k) \in \bbR^{d \times d}$ are the gradient and the Hessian matrix of $f_\theta(k)$ with respect to $\theta$, respectively.

\begin{theorem}
	\label{thm-param-clt}
	In the model $\{ f_\theta : \theta \in \Theta \}$ stated in Section~\ref{SectionModel} with compact parameter
	space $\Theta\subset \bbR^d $ and \pa functions satisfying $f_\q(k)\le C k^\b$ for some  constants  $C$ and $\beta<1$ or
	$f_\q(k)=k+\a$ for some constant $\a>-1$, for every $k$,  and  every $\q\in\Theta$, assume in addition
	that the  \pa functions satisfy, for some constants $\gamma > 0$ and $C > 0$,
	\begin{align}
		\bigl\|\dot f_\theta(k)\bigr\|+\bigl\| \ddot f_\theta(k)\bigr\|                                    & \le C k \log^\gamma k, \label{eqn-max-partial}    \\
		\Bigl\|\frac{\dot f_\theta}{f_\theta}(k)\Bigr\|+ \Bigl\|\frac{\ddot f_\theta}{f_\theta}(k) \Bigr\| & \le C \log^\gamma k. \label{eqn-partial-quotient}
	\end{align}
	Then  as $n \rightarrow \infty$, for any consistent sequence of solutions $\hat\theta_n$ to the likelihood equations
	$\dot{\iota}_n(f_\q)=0$,  the sequence $\sqrt{n} (\hat{\theta}_n - \theta_0) $ converges in distribution to the
	$N(0, V_0^{-1})$ distribution, for $V_0$ the ($d\times d$) matrix given by
	\begin{equation}
		V_0 =\sumk  \frac{ \dot f_{\theta_0} \dot f_{\theta_0}^T}{ f_{\theta_0}^2}(k) \qkzero
		- \Bigl(\sumk \frac{\dot f_{\theta_0}}{ f_{\theta_0} }(k) \qkzero \Bigr) \Bigl(\sumk \frac{\dot f_{\theta_0}^T}{ f_{\theta_0} }(k) \qkzero\Bigr).
		\label{eqn:V0}
	\end{equation}
	In particular this is true for the maximum likelihood estimator if $\q_0$ is identifiable and interior to $\Theta$.
\end{theorem}

\begin{corollary}
	For the parametric family $f_{\alpha,\beta}(k) = (k + \alpha)^\beta$ with true parameter
	$(\alpha_0,\beta_0)$ in the interior of the parameter set
	$\Theta_\varepsilon = [-1 + \varepsilon, \varepsilon^{-1}] \times [0, 1]$ for some small $\varepsilon>0$,
	the \mle $(\hat{\alpha}_n, \hat{\beta}_n)$ satisfies, as  $n \rightarrow \infty$,
	\begin{equation*}
		\sqrt{n} \left( \begin{bmatrix}
				\hat{\alpha}_n \\
				\hat{\beta}_n
			\end{bmatrix}
		-
		\begin{bmatrix}
				\alpha_0 \\
				\beta_0
			\end{bmatrix}\right)
		\rightsquigarrow N(0, V_{\alpha_0, \beta_0}^{-1}),
	\end{equation*}
	where $V_{\alpha_0, \beta_0}$ is defined as the negative of $\ddot{\iota}(f_{\alpha_0, \beta_0})$ in \eqref{eqn-hessian-iota-alpha-beta}.
	\label{cor:k-alpha-beta}
\end{corollary}


To test whether the \pa function is affine we may compare an estimator $\hat\b_n$ for $\beta$ in the
model considered in Corollary~\ref{cor:k-alpha-beta} to the value $\b_0=1$ in the affine case. The
Wald-type test statistic admits the form
\begin{equation*}
	T_n := \frac{{n}^{1/2}(\hat{\beta}_n  -1) }{(V_{\hat{\alpha}, 1}^{-1})_{2,2}^{1/2}},
\end{equation*}
where $V_{\hat\alpha, 1}$ is obtained by plugging in $(\hat{\alpha}_n,1)$ in the definition of
$V_{\alpha, \beta}$ in \eqref{eqn-hessian-iota-alpha-beta} and  $(V_{\hat{\alpha}, 1}^{-1})_{2,2}$ is the $(2,2)$-element of $V_{\hat{\alpha}, 1}^{-1}$.

\begin{theorem}
	For the parametric family $f_{\alpha,\beta}(k) = (k + \alpha)^\beta$ with true parameter
	$(\alpha_0,1)$ in the parameter set $\Theta_\varepsilon = (-1 + \varepsilon, \varepsilon^{-1}) \times [0, 1]$
	for some small $\varepsilon>0$, the sequence $T_n$ tends in distribution to $Z\bbind_{Z\le 0}$, for $Z$
	a standard normal variable.
	\label{cor:affine-test}
\end{theorem}

\section{A Remedy to the History Problem}
\label{sec:remedy-history}
A practical problem with the \mle is that the log likelihood function \eqref{eqn:iota-nn} and its derivative \eqref{eqn:iota-n}
depend on the history of the network evolution.  In many real-world applications
observing the entire history is impossible or too costly, and only the final snapshot at time $n$ is
available. For instance, when building a social network model, we may observe the final network,
but recovering how it exactly evolved into its current shape is difficult---we would need to check with everyone when
s/he became friends with everyone else and establish a strict time order.

This problem is solvable. The log likelihood \eqref{eqn:iota-nn} or its derivative \eqref{eqn:iota-n} consist
of two terms, and the history problem only arises in the second term, which results from the norming constant to
the likelihood. The first terms in \eqref{eqn:iota-nn} or \eqref{eqn:iota-n}
depend on the network only through $P_{>k}(n)$ and hence are available from
the final snapshot. The second terms
are the C\'esaro averages of $\log S_{f_\q}(t-1)$ and its derivative $S_{\dot f_\q}(t-1)/S_{f_\q}(t-1)$, respectively.
Because $S_h(t)/t$ tends to the limit $\sum_{k=1}^\infty h(k)p_k$ almost surely, as $t\rightarrow\infty$, these
C\'esaro averages are  asymptotically actually very close to $\log S_{f_\q}(n)/n$ and  $S_{\dot f_\q}(n)/S_{f_\q}(n)$,
which do depend only on the snapshot of the network at time $n$. The remedy is to
replace \eqref{eqn:iota-nn} or \eqref{eqn:iota-n} by
\begin{align}
	\tilde \iota_n(f_\theta) & = \sum_{k=1}^\infty \log f_\theta(k) P_{>k}(n) -  \log S_{f_\q}(n),
	\label{eqn:tilde-iota}
	\\
	\dot{\tilde{\iota}}_n(f_{\theta})
	                         & = \sum_{k=1}^\infty \frac{\dot{f}_\theta}{f_\theta}(k) P_{>k}(n) - \frac{S_{\dot f_\q}(n)}{S_{f_\q}(n)}. 
	\label{eqn:tilde-iota-dot}
\end{align}
Define a \qmle $\tilde\q_n$ as the maximizer of the first function or a zero of the second.
Inspection of the proof of Theorem~\ref{thm-consistency-param} readily shows that this \qmle is consistent
under the same conditions as the \mle.

\begin{theorem}
	Under the conditions of Theorem~\ref{thm-consistency-param},
	the \qmle $\tilde{\theta}_n$ is consistent, i.e., as $n \rightarrow \infty$,
	$\tilde{\theta}_n \rightarrow \theta_0$ almost surely, under $\theta_0$.
	\label{thm:qmle-consistency}
\end{theorem}

We have no proof of the asymptotic normality of the \qmle in the same generality as
for the maximum likelihood estimator, but we note the following general theorem, and its corollary.

Let $F_\q(k)=\bigl(\sum_{j=1}^{k-1} \dot f_\q/f_\q(j),\dot f_\q(k),f_\q(k)\bigr)^T$.

\begin{theorem}
	\label{TheoremANQMLE}
	Assume the conditions of Theorem~\ref{thm-consistency-param}, and in addition
	assume that the sequence of random vectors $\sum_{k=1}^\infty F_{\q_0}(k)\sqrt n(P_k(n)-p_k^{\sss (0)})$
	is asymptotically normal with mean zero and covariance $W_0$ and that $\q_0$ is interior to $\Theta$. Then the sequence of pseudo maximum
	likelihood estimators $\tilde\theta_n$ satisfies
	$\sqrt n (\tilde\theta_n-\q_0)\weak N(0, V_0^{-1}VV_0^{-1})$, under $\theta_0$ as $n \rightarrow \infty$,
	where  $V$ is given in the proof below and $V_0$ is given in \eqref{eqn:V0}.
\end{theorem}

It is plausible that the sequence of empirical degrees $\sqrt n (P_k(n)-p_k^{\sss (0)})$ is asymptotically normally distributed
in some generality,
but this has been established  only for the affine \pa function (see \cite{mori2002random}
and \cite{resnick2016asymptotic}, or Proposition~\ref{prop:asymp-normality-affine}). The preceding theorem requires that certain
linear combinations of the variables over $k$ are asymptotically normal, where the coefficients
typically tend to infinity with $k$ (e.g.\ at the order $k\log k$). Although this convergence requires
additional bounds for large $k$, the condition of the theorem seems plausible in general, for sublinear \pa models.

In the appendix we verify the condition for the interesting case of \pa functions
that are eventually constant, for which the limiting degree distribution
follows a power law with exponential cut-off (\citet{rudas2007random}). This leads to
the following corollary.

\begin{corollary}
	\label{thm:pmle-asymp-normal-trunc}
	In the model with \pa functions $f_\q$ satisfying $f_\q(k)=f_\q(k\wedge K)$, for some
	given $K\in\bbN_+$ for which the eigenvalue condition $\Re \lambda_2(A_K) < \lambda_1(A_K)/2$ (as in the
	appendix)  holds at $\theta_0$,
	the \qmle $\tilde{\theta}_n$ satisfies $\sqrt n(\tilde\q_n-\q_0)\weak N(0, V_0^{-1} V V_0^{-1})$,
	provided that it is consistent at $\q_0$.
\end{corollary}

The limiting covariance in the preceding theorem and corollary may be complicated.
To get around this, we propose the following bootstrap procedure.
Suppose that we observe the final snapshot of the network $G_{n}$ with $n$ nodes.
\begin{enumerate}
	\item Obtain the \qmle $\tilde{\theta}_{n}$ based on $G_n$.
	\item Given $\tilde{\theta}_{n}$, simulate \pa networks $(G_{m}^{(i)})_{i=1}^s$ with \pa function $f_{\tilde{\theta}_{n}}$, each
	      with $m$ nodes.
	\item Obtain the \qmle $\tilde{\theta}_{m}^{(i)}$  based on $G_{m}^{(i)}$, for $i \in [s]$.
	\item Approximate the limit variance of $\sqrt{n} (\tilde{\theta}_{n} - \theta_0)$ by
	      \begin{equation}
		      \tilde{\Sigma}_{m, s}(f_{\tilde{\theta}_n}) : = {m} \biggl\{ \frac{1}{s} \sum_{i=1}^s \tilde{\theta}_m^{(i)} \bigl( \tilde{\theta}_m^{(i)} \bigr)^T - \Bigl( \frac{1}{s} \sum_{i=1}^s \tilde{\theta}_m^{(i)}  \Bigr)\Bigl( \frac{1}{s} \sum_{i=1}^s \tilde{\theta}_m^{(i)}  \Bigr)^T  \biggr\}.
		      \label{eqn:tilde-Sigma}
	      \end{equation}
\end{enumerate}
This procedure will be consistent under a mild continuity condition on the model $\q\mapsto f_\q$.

Section 7 of \cite{gao2017asymptotic} on the affine \pa model addressed a different history problem---the history of the
initial degrees. In the present paper  we study the case of fixed initial degree $1$,
but need to know how the incoming nodes connect, a problem that did not arise in the affine model.
The empirical estimator considered in \cite{gao2017consistent} is curiously free of the history problem.

\section{Connecting the \ee and the \qmle}
\label{sec:connect-ee-qmle}
Write $f(k)$ as $\theta_k$ and suppose that we are interested in estimating the
\emph{infinite-dimensional} vector $(\theta_k)_{k=1}^\infty$.
The pseudo log-likelihood function \eqref{eqn:tilde-iota} then takes the form
\[
	\tilde{\iota}_n(\theta) = \sumk (\log \theta_k) P_{>k}(n) - \log S_\theta(n),
\]
where
\(
S_\theta(t) = \sumk \theta_k N_k(t)
\) is the total preference.
Taking the derivative with respect to $\theta_k$, we obtain
\begin{equation*}
	\frac{\partial\tilde{\iota}_n(\theta)}{\partial \theta_k} = \frac{P_{>k}(n)}{ \theta_k} - \frac{N_k(n)}{ \sumj \theta_j N_j(n)}.
\end{equation*}
Setting this equation to zero for every $k$, we obtain the system of equations
(for simplicity assume that $N_k(n)>0$ for every $k$)
\begin{equation*}
	\frac{\theta_k}{ \sumj \theta_j P_j(n)} = \frac{N_{>k}(n)}{N_k(n)},\qquad k \in \bbN_+.
	\label{eqn:ee-qmle}
\end{equation*}
The right side is the aforementioned empirical estimator $\rkn$, defined in \cite{gao2017consistent}.
The \pa function $f$ or parameter $\q_k$ is identifiable up to a scale factor only. We conclude that
the \ee is the \qmle if we do not impose any parametric assumption
and wish to estimate $\theta_k$ individually for any $k$.

\section{Numerical Illustrations}
\label{sec:simu}
In this section we numerically study the performance of the \mle, the \qmle,  the empirical estimator,  the Wald test and
the bootstrap estimator. We
simulated data using the following three examples of a \pa function, each of the type $f(k)=(k+\a)^\b$:
\begin{alignat*}{5}
	 & f^{\sss (2)}(k) &  & = k^{2/3},       &  & \qquad &  & \a=0, &  & \ \b=2/3,  \\
	 & f^{\sss (4)}(k) &  & = (k + 4)^{4/5}, &  & \qquad &  & \a=4, &  & \ \b=4/5,  \\
	 & f^{\sss (5)}(k) &  & = k+2,           &  & \qquad &  & \a=2, &  & \ \beta=1.
\end{alignat*}
In every setting we conducted  $N = 1000$ repetitions of the experiment in which we simulated
a \pa tree of $n$ nodes, for varying $n$, and computed the estimators and/or test. For every estimator $\hat\theta_i$ of
$\theta=(\a,\b)$ we computed the \textit{sample mean difference} $(1/N) \sum_{i=1}^N (\hat\theta_i - \theta_0)$
and the rescaled \textit{sample covariance} matrix $(n/N)\sum_{i=1}^N (\hat\theta_i - \theta_0) (\hat\theta_i - \theta_0)^T$.
According to our theory, for large $n$ the first should be close to zero, and the second should
be close to the deterministic matrix $V_0^{-1}$, for $V_0$ given in \eqref{eqn:V0}, which depends on the \pa function.

\subsection{MLE}
\label{sec:simu-mle}
The limiting covariance matrix $V_0^{-1}$ of $\sqrt{n} (\hat\theta_n - \theta_0)$  under the \pa function
$f^{\sss (2)}$ is computed to be
\begin{equation*}
	\begin{pmatrix}
		169.30\  & 47.56 \\
		47.56    & 14.94
	\end{pmatrix}.
\end{equation*}
Table~\ref{tab:sample-mean-covariance-evolution} gives the results of our simulation experiments,
where we simulated \pa trees of three different sizes: $n= 10^4$, $n=10^5$ or $n=10^6$ nodes.
As expected the sample mean difference decreases with $n$, while the sample covariance swings around
the expected limit.

\begin{table}[htp]
	\centering
	\begin{tabular}{c c c}
		\toprule
		\# of nodes & Sample Mean Difference & Sample Covariance \\
		\midrule
		$10^4$      &
		\begin{tabular}{c c}
			7.29e-03 & 3.24e-05
		\end{tabular}
		            &
		\begin{tabular}{c c}
			173.19 & 48.14 \\
			48.14  & 15.08
		\end{tabular}
		\\
		\midrule
		$10^5$      &
		\begin{tabular}{c c}
			5.71e-04 & -7.03e-05
		\end{tabular}
		            &
		\begin{tabular}{c c}
			167.89 & 47.23 \\
			47.23  & 15.03
		\end{tabular}
		\\
		\midrule
		$10^6$      &
		\begin{tabular}{c c}
			1.71e-04 & -2.59e-05
		\end{tabular}
		            &
		\begin{tabular}{c c}
			163.02 & 46.60 \\
			46.60  & 14.97
		\end{tabular}                         \\
		\bottomrule
	\end{tabular}
	\caption{Sample mean difference and covariance of the maximum likelihood estimator of the parameters $\a$ and $\b$
		of the \pa function $f(k)=(k+\a)^\b$ for trees of  three sizes of node sets $n$ generated according
		to \pa function $f^{\sss (2)}$.}
	\label{tab:sample-mean-covariance-evolution}
\end{table}

We conducted the same experiment with \pa functions $f^{\sss (4)}$ and $f^{\sss (5)}$, but for trees of a single
size of $n=10^6$ nodes. The results presented in Table~\ref{tab:sample-mean-covariance-different}
again confirm the theory.

\begin{table}[htp]
	\centering
	\begin{tabular}{c c c c}
		\toprule
		\pa Function \                    & Sample Mean Difference & Sample Covariance & Limit Variance \\
		\midrule
		$f^{\sss (4)}$                    &
		\begin{tabular}{c c}
			-2.43e-02 & -2.72e-03
		\end{tabular}
		                                  &
		\begin{tabular}{c c}
			42764.75 & 4743.46 \\
			4743.46  & 540.71
		\end{tabular} &
		\begin{tabular}{c c}
			42429.33 & 4716.76 \\
			4716.76  & 539.75
		\end{tabular}
		\\
		\midrule
		$f^{\sss (5)}$                    &
		\begin{tabular}{c c}
			-5.35e-04 & -1.01e-04
		\end{tabular}
		                                  &
		\begin{tabular}{c c}
			1817.94 & 325.66 \\
			325.66  & 63.07
		\end{tabular}
		                                  &
		\begin{tabular}{c c}
			1762.05 & 316.58 \\
			316.58  & 61.64
		\end{tabular}                                                                \\
		\bottomrule
	\end{tabular}
	\caption{Sample mean difference and covariance of the maximum likelihood estimator of the parameters $\a$ and $\b$
	of the \pa function $f(k)=(k+\a)^\b$  for trees of $n=10^6$ nodes generated
	according to the \pa functions $f^{\sss (4)}$ and $f^{\sss (5)}$. The last column gives
	the covariance matrix $V_0^{-1}$.}
	\label{tab:sample-mean-covariance-different}
\end{table}

\subsection{Comparing the estimators}
We compared the performance of the empirical estimator \eqref{eqn:ee-equation}, the \qmle and the \mle on samples of trees of $n=10^6$ nodes generated using the \pa function $f^{\sss (2)}$.
Table~\ref{tab:sample-mean-covariance-f2-ee-pmle} gives numerical results, while Figure~\ref{fig:qqplot-all}
shows QQ-plots of the estimators. The first line of Table~\ref{tab:sample-mean-covariance-f2-ee-pmle} repeats
the relevant (third) line of Table~\ref{tab:sample-mean-covariance-evolution}.

The comparison is particularly interesting as the theoretical variances of the empirical estimator and the \qmle are unknown.
In our experiment the sample covariance of the \qmle is only twice bigger than that of the \mle. In contrast,
the sample covariance of the empirical estimator is larger than that of the other two estimators by an order of magnitude.
We conclude that it helps to use a parametric model, if this can be correctly specified.

As is evident from Figure~\ref{fig:qqplot-all}, the \mle and \qmle are asymptotically normal. The same seems
to be true for the empirical estimator, with the largest visible possible deviation in the left tail of its distribution.

\begin{table}[htp]
	\centering
	\begin{tabular}{c c c}
		\toprule
		Estimator & Sample Mean Difference & Sample Covariance \\
		\midrule
		MLE       &
		\begin{tabular}{c c}
			1.71e-04 & -2.59e-05
		\end{tabular}
		          &
		\begin{tabular}{c c}
			163.02 & 46.60 \\
			46.60  & 14.97
		\end{tabular}                       \\
		\midrule
		EE        &
		\begin{tabular}{c c}
			1.02e-03 & 2.58e-04
		\end{tabular}
		          &
		\begin{tabular}{c c}
			6840.42 & 2965.18 \\
			2965.18 & 1297.15
		\end{tabular}                       \\
		\midrule
		PMLE      &
		\begin{tabular}{c c}
			-1.05e-03 & -3.79e-04
		\end{tabular}
		          &
		\begin{tabular}{c c}
			297.11 & 85.40 \\
			85.40  & 26.20
		\end{tabular}                       \\
		\bottomrule
	\end{tabular}
	\caption{Comparison of the \mle, the empirical estimator and the \qmle for estimating the parameters $\a$ and $\b$
		of the \pa function $f(k)=(k+\a)^\b$  based on  trees with $n=10^6$ nodes generated
		according to \pa function $f^{\sss (2)}$.}
	\label{tab:sample-mean-covariance-f2-ee-pmle}
\end{table}

\begin{figure}[htp]
	\centering
	\begin{subfigure}[b]{0.48\textwidth}
		\includegraphics[width=\textwidth]{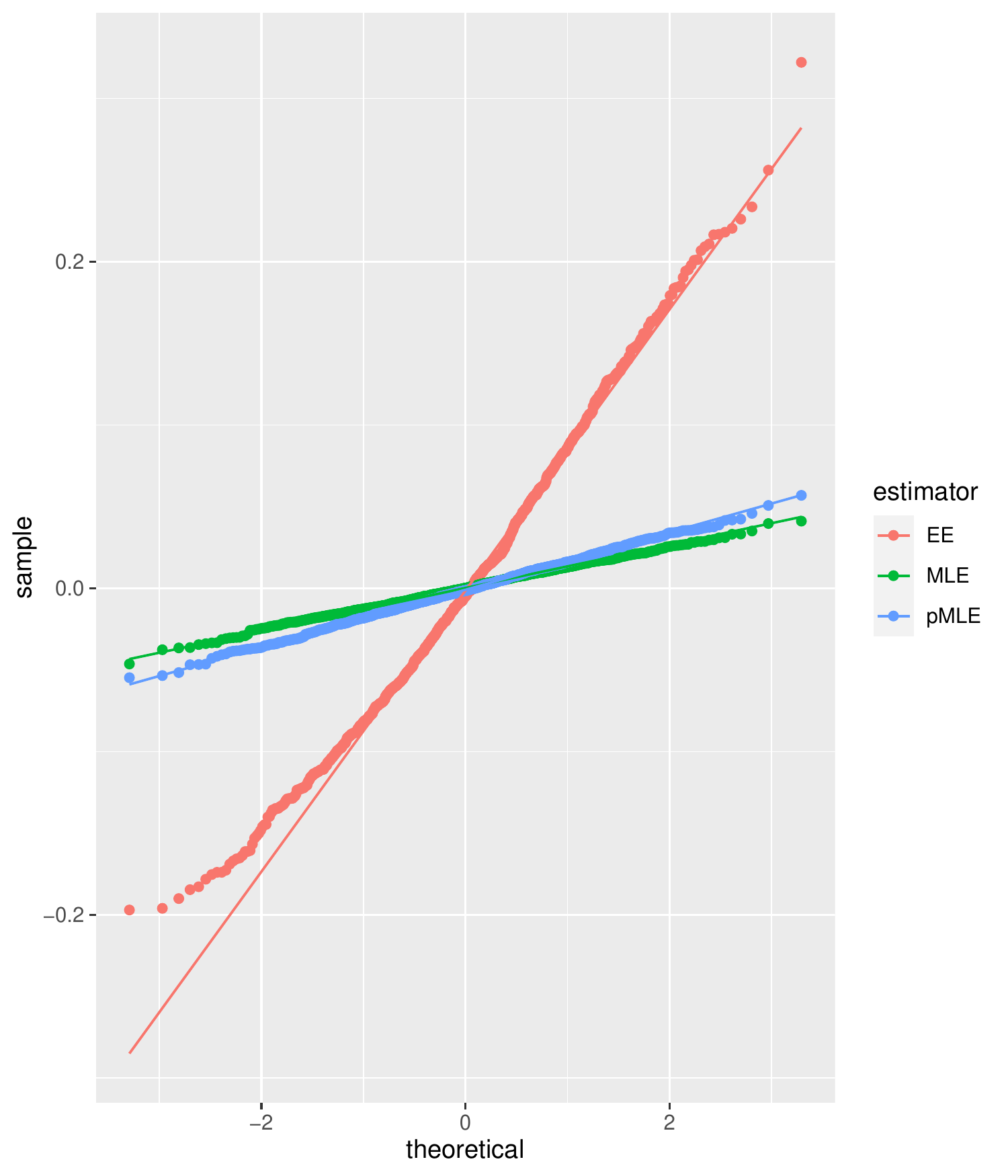}
		\caption{Estimators of $\alpha$}
		\label{fig:alpha}
	\end{subfigure}
	~ 
	\begin{subfigure}[b]{0.48\textwidth}
		\includegraphics[width=\textwidth]{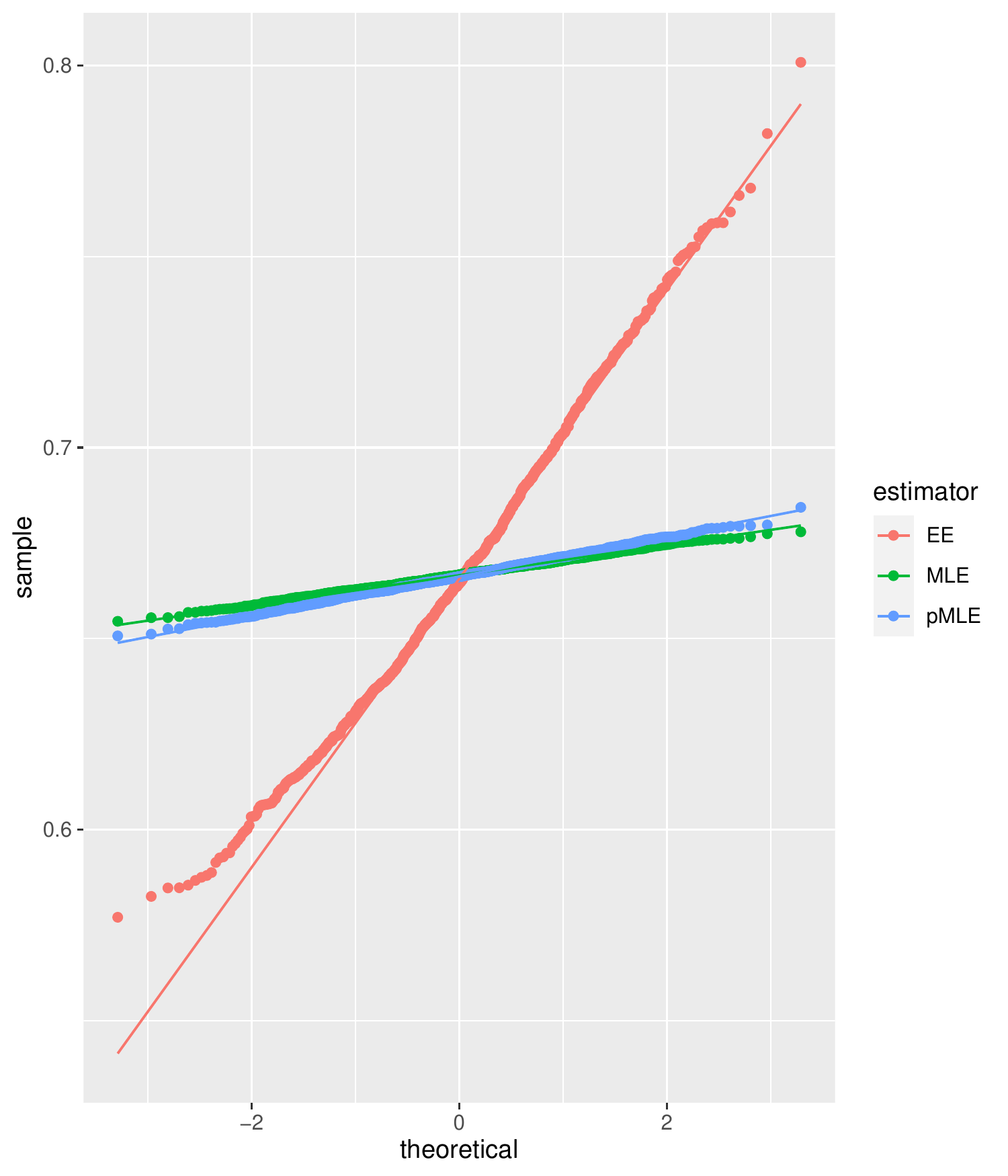}
		\caption{Estimators of $\beta$}
		\label{fig:beta}
	\end{subfigure}
	\caption{\it QQ-plots of 1000 realizations of the \mle, the empirical estimator and the \qmle of the parameters $\a$ and $\b$
		of the \pa function $f(k)=(k+\a)^\b$  based on trees with $n=10^6$ nodes generated according
		to \pa function $f^{\sss (2)}$.}
	\label{fig:qqplot-all}
\end{figure}

\subsection{The Wald-type test for affinity}
\label{sec:wald-test-simu}
We applied the Wald-type test for affinity given in Corollary~\ref{cor:affine-test} on \pa trees generated using
$f^{\sss (4)}$ and $f^{\sss (5)}$.  The nominal size of the tests was set at 0.05. In both cases we
registered the proportion of repetitions in which the null hypothesis was rejected.
As $f^{\sss (5)}$ is affine, rejection constitutes a type-I error in this case, while
for $f^{\sss (4)}$ the proportion of rejections is a measure of the test's power.
The results are summarized in Table~\ref{tab:wald-test}.
The expected proportion of type-I errors for $f^{\sss (5)}$ was 5.2\%, close to the nominal value.
The power of the test at $f^{\sss (4)}$ was outright 1, resulting from the fact that the number of nodes was large
and $f^{\sss (4)}$ is far from affine.

\begin{table}[htp]
	\centering
	\begin{tabular}{c c c}
		\toprule
		\pa function \  & Type-I error & Power \\
		\midrule
		$f^{\sss (4)}$  & -            & 1     \\
		$f^{\sss (5)}$  & .052         & -     \\
		\bottomrule
	\end{tabular}
	\caption{Proportion of rejections of the Wald-type affinity test
		with nominal size  0.05 based on \pa trees  with $n=10^6$ nodes for two \pa functions.}
	\label{tab:wald-test}
\end{table}

\subsection{The \qmle with bootstrapped variance}
\label{sec:bootstrapped-test-simu}
We conducted two experiments to illuminate the bootstrap procedure
for the \qmle in Section~\ref{sec:remedy-history}, applied to the model
$f(k)=(k+\a)^\b$.   In both cases the  data was generated under the \pa function $f^{\sss (2)}$, corresponding to $\a=0$ and $\b=2/3$.
Furthermore, the bootstrap sample size in \eqref{eqn:tilde-Sigma} was set
to $m= 10^5$ and the number of bootstrap replicates to $s = 10^3$.

In the first experiment we performed Wald type tests for the hypotheses $H_0: \a=0$ and $H_0:\b=2/3$ using
the relevant coordinate of the \qmle $(\tilde\a_n,\tilde\b_n)$ with its variance estimated by the bootstrap estimator \eqref{eqn:tilde-Sigma}.
In the case of the hypothesis $H_0: \a=0$ this entails the test statistic
\begin{equation*}
	T_{n,m,s} : = \frac{n^{1/2} \, \tilde\alpha_n }{ \bigl( \tilde\Sigma_{m,s} ( f_{\alpha_0, \tilde\beta_n}) _{1,1}\bigr)^{1/2}}.
\end{equation*}
The test for  $H_0: \b=2/3$ is similarly based on a standardized version of $\tilde\b_n$.
We rejected the null hypothesis when $|T_{n,m,s}| > z_{0.025}$, corresponding to the working hypothesis that $T_{n,m,s} \sim N(0,1)$ under the null hypothesis and nominal size $0.05$.
Table~\ref{tab:wald-test-t} gives the proportions of rejections in $N =1000$ repetitions of the experiment.

\begin{table}[htp]
	\centering
	\begin{tabular}{c c}
		\toprule
		$H_0$         & Type-I error \\ 
		\midrule
		$\alpha = 0 $ & 0.047        \\
		$\beta = 2/3$ & 0.063        \\
		\bottomrule
	\end{tabular}
	\caption{Proportions of rejections of the size-0.05 Wald-type tests with bootstrapped variance
		of the null hypotheses $H_0: \a=0$ and $H_0: \b=2/3$ on the \pa function $f(k)=(k+\a)^\b$  for trees with $n=10^6$ nodes
		generated according to \pa function $f^{\sss (2)}$.
		The bootstrap sample size was set to $m= 10^5$ and the number of bootstrap replicates to $s = 10^3$.}
	\label{tab:wald-test-t}
\end{table}

In the second experiment we simulated $N=1000$ replicates of the normalized and projected \qmle
\[
	d_n := r_n^T \sqrt{n} \bigl( \tilde{\Sigma}_{m,s}(f_{\tilde{\alpha}_n,\tilde{\beta}_n} )\bigr)^{-1/2}  \binom{\tilde{\alpha}_n}{\tilde{\beta}_n},
\]
where $r_n$ are i.i.d.\ random vectors from the unit circle. Figure~\ref{fig:bootstrap-qqplot} shows a QQ-plot of these
1000 values against the standard normal distribution.

\begin{figure}[htp]
	\centering
	\includegraphics[width=0.7\textwidth]{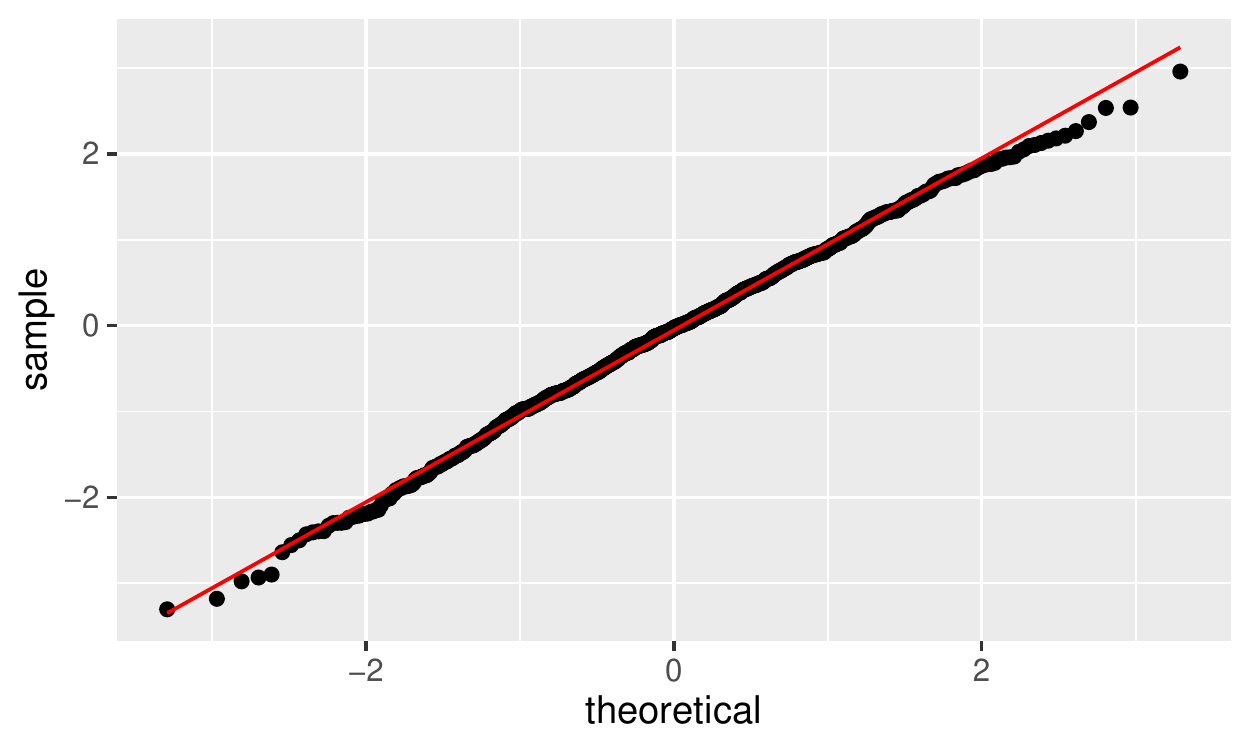}
	\caption{\it QQ-plot of 1000 realizations of the randomly projected pseudo maximum likelihood estimators
		in the model $f(k)=(k+\a)^\b$ normalized by their bootstrapped variances
		based on trees with $n=10^6$ nodes generated according to $f^{\sss (2)}$.
		The bootstrap sample size was set to $m= 10^5$ and the number of bootstrap replicates to $s = 10^3$.}
	\label{fig:bootstrap-qqplot}
\end{figure}

Both experiments support the practical use of asymptotic normal limit theory for the \qmle and validate
the bootstrap procedure for estimating the asymptotic variance.

\section{Proofs of main results}
\label{sec:proofs}

\begin{proof}[Proof of Proposition~\ref{lemma-conv-phi1-phi2}]
	We apply the general results due to \cite{jagers1975branching} and \cite{nerman1981convergence}, as summarized in
	the appendix, Section~\ref{sec:rooted-tree}.

	The events of the pure birth process \eqref{eqn-birth-process} can be represented as $T_1<T_1+T_2<T_1+T_2+T_3<\cdots$,
	for independent exponential random variables $(T_k)_{k=1}^\infty$ with rates $(f(k))_{k=1}^\infty$.
	The total number of births $\xi(t)=\int \bbind_{(0,t]}(u)\,\xi(du)$ at time $t$ is equal to $\sum_{l=1}^\infty \bbind_{(0,t]}(T_1+\cdots+T_l)$,
	which tends to infinity almost surely as $t\ra\infty$ by the assumption that the \pa function is bounded from below by $f(1)>0$.
	The birth times are clearly not restricted to any lattice, and the functions $t\mapsto \bbE[\varphi_i(t)]$ are continuous almost
	everywhere in view of their monotonicity. Thus, it suffices to show that the Malthusian parameter exists and to verify
	conditions \eqref{eqn:V-beta} and \eqref{eqn:beta-xi}.

	Since $\int \mathrm{e}^{-\lambda u}\,\xi(du)=\sum_{l=1}^\infty \mathrm{e}^{-\lambda(T_1+\cdots+T_l)}$,
	\begin{align*}
		\bbE\Bigl[ \int_0^\infty \mathrm{e}^{-\lambda u}\,\xi(du)\Bigr]   & =\bbE \Bigl[\sum_{l=1}^\infty \mathrm{e}^{-\lambda(T_1+\cdots+T_l)}\Bigr]
		=\sum_{l=1}^\infty \prod_{i=1}^l \frac{f(i)}{\lambda+f(i)},                                                                                                                                                              \\
		\bbE\Bigl[ \int_0^\infty \mathrm{e}^{-\lambda u} \xi(du) \Bigr]^2 & = \bbE \Bigl[ \sum_{l=1}^\infty \mathrm{e}^{- \lambda (T_1 + \dots + T_l)} \Bigr]^2
		= \sum_{k=1}^\infty \sum_{l= 1}^\infty \bbE\bigl[ \mathrm{e}^{-\lambda(T_1 +\cdots + T_k) - \lambda (T_1 + \cdots + T_l)}\bigr]                                                                                          \\
		                                                                  & =  \rho_f(2\lambda)+2 \sum_{k=1}^\infty \sum_{l=k+1}^\infty\prod_{i=1}^k \frac{f(i)}{2\lambda + f(i)} \prod_{i=k+1}^l \frac{f(i)}{ \lambda + f(i)} .
	\end{align*}
	The left side of the first formula is $\int_0^\infty \mathrm{e}^{-\lambda u}\,\mu(du)=\rho_f(\lambda)$, and hence this
	formula verifies equation \eqref{laplace-transform}.

	First consider the strictly sublinear case, where  $f(k) \le C k^{\beta}$. Because both expressions in the display are monotone
	in $f$, we obtain upper bounds on these expressions by evaluating their right sides for  $f(k) = C k^{\beta}$. Because $\log(1+x)\ge x/2$ for
	$x\in [0,1]$, we have $f(i)/(\lambda+f(i))\le \exp(-\lambda/(2f(i)))$, for $i\ge I_\lambda:= (\lambda/C)^{1/\beta}$.
	Since the quotients and their products are also bounded by 1,
	$$\sum_{l=1}^\infty \prod_{i=1}^l \frac{f(i)}{\lambda+f(i)}\le I_\lambda+\sum_{l>I_\lambda}e ^{-\sum_{I_\lambda<i\le l} \lambda/(2C i^{\beta})}.$$
	Because $\beta<1$, the sum in the exponent is of the order $l^{1-\beta}$ as $l\rightarrow \infty$, and hence the series
	is finite for every $\lambda>0$. The function $\rho_f$ is monotone decreasing and tends to zero as $\lambda\rightarrow\infty$,
	by the dominated convergence theorem, and to infinity as $\lambda\downarrow 0$, by the monotone convergence theorem.
	It follows that the Malthusian parameter $\malt$ exists and is contained in $(0,\infty)$. We also
	conclude that \eqref{eqn:V-beta} is satisfied, for any $\lambda\in (0,\malt)$.

	The second moment of $\int_0^\infty \mathrm{e}^{-\lambda u}\,\xi(du)$ is also finite for every $\lambda>0$. To see this, it suffices
	to bound the double sum on the right. By the same arguments, for $k>I_\lambda$,
	\begin{align*}
		\sum_{l=k+1}^\infty \prod_{i=k+1}^l \frac{f(i)}{\lambda + f(i)}
		 & \le \sum_{l= k + 1}^\infty \mathrm{e}^{-\sum_{i=k+1}^ l \lambda/(2C i^{\beta})}
		\le \sum_{l= k + 1}^\infty \mathrm{e}^{ - (\lambda/2C(1-\beta)) [l^{1-\beta} - (k+1)^{1-\beta}]} .
	\end{align*}
	This is bounded from above by a constant for every $k$ (in fact by a multiple of $k^{-\beta}$).
	Inserting this bound in the double sum, we are left with
	a single sum of the same type as before, except that $\lambda$ is replaced by $2\lambda$.
	We conclude that the second moment of $\int_{0}^\infty \mathrm{e}^{-\lambda u} \xi(du)$ is finite for every $\lambda>0$ as well.

	We have $\mathrm{e}^{-\lambda t} \xi(t)  = \int_{0}^t \mathrm{e}^{-\lambda t} \xi(du) \le \int_0^t \mathrm{e}^{-\lambda u} \,\xi(du)$, for every $t$ and $\lambda>0$. Combined
	with the assumption $\varphi_i(t)\le C\xi(t)^2$,
	we see that $\mathrm{e}^{-\lambda t}\varphi_i(t)\le C\bigl[\int_0^t \mathrm{e}^{-\lambda u/2}\,\xi(du)\bigr]^2$. It follows
	that $\bbE \sup_{t>0} \bigl(\mathrm{e}^{-\lambda t} \varphi_i(t)\bigr)\le C\bbE \bigl[\int_0^\infty \mathrm{e}^{-\lambda u/2}\,\xi(du)\bigr]^2<\infty$, for every $\lambda>0$,
	thus verifying \eqref{eqn:beta-xi}.

	In the case that $f(k)=k+\alpha$, for $\alpha > -1$, the function $\rho_f$ takes the form
	$\rho_f(\lambda) = (1+\alpha)/(\lambda-1)$ (see \cite{rudas2007random})  and hence the Malthusian parameter is $\malt = 2+ \alpha$ and
	the function $\rho_f$ is finite for $\lambda>1$. We conclude that \eqref{eqn:V-beta} is satisfied.
	We show below that the second moment of $\int_{0}^\infty \mathrm{e}^{-\lambda u} \xi(du)$ is finite for every $\lambda>1$ as well.
	If  $\varphi_i(t)\le C \xi(t)^r$, then $\sup_{t>0} \mathrm{e}^{-\lambda t}\varphi_i(t)\le C\bigl[\int_0^\infty \mathrm{e}^{-\lambda u/r}\,\xi(du)\bigr]^r$,
	which then has a finite moment for any $\lambda$ such that $\lambda/r>1$ for some $r\le 2$.
	Since $r$ can be chosen such that  $r\le 2$ and $r<2+\alpha$, where $2+\alpha>1$,  there exists $\l\in(1,\malt)$ that
	satisfies these restrictions, and hence \eqref{eqn:beta-xi} follows.

	For  $i\rightarrow\infty$, we have $f(i)/(\l+f(i))= \exp(-\log(1+\lambda/(i+\alpha))\asymp \exp(-\l/(i+\alpha))$. Hence,
	for $k\ra\infty$, and $\l>1$,
	\begin{align*}
		\sum_{l=k+1}^\infty \prod_{i=k+1}^l \frac{f(i)}{\lambda + f(i)}
		                                           & \asymp  \sum_{l=k+1}^\infty \mathrm{e}^{-\lambda \log\frac{l+\alpha}{ k + \alpha}}
		=\sum_{l=k+1}^\infty \left( \frac{k+\alpha}{l+\alpha} \right)^\lambda
		\asymp k,                                                                                                                       \\
		\prod_{i=1}^k \frac{f(i)}{2\lambda + f(i)} & \asymp \mathrm{e}^{-2\lambda\log (k+\a)}\asymp k^{-2\lambda}.
	\end{align*}
	Therefore, the double sum in the second moment of $\int_{0}^\infty \mathrm{e}^{-\lambda u} \xi(du)$ is bounded by
	a multiple of $\sum_{k=1}^\infty k^{-(2\lambda-1)}$, which is finite for $\lambda>1$.
\end{proof}

\begin{proof}[Proof of Corollary~\ref{lemma-hk-pk-conv}]
	The  random characteristics $\varphi_1(t) =h(\xi(t)+1) $  and $\varphi_2(t) = \bbind_{\{ t\ge 0\}}$ satisfy the conditions
	of Proposition~\ref{lemma-conv-phi1-phi2}, and hence $Z_t^{\varphi_1}/Z_t^{\varphi_2}$ converges almost surely to
	the limit given in the proposition, as $t\rightarrow\infty$. If
	$\tau_n$ is the first time that the \pa tree $\Upsilon_t$ possesses $n$ nodes, then $\tau_n\rightarrow\infty$ almost surely,
	as the total number of individuals at any given $t$ is finite, almost surely. (At any given (finite) time, every individual has a finite number of offspring
	and there can be at most finitely many generations, since a new generation can be formed not faster than an exponential waiting time with mean $1/f(1)$.)
	Hence, the sequence $Z_{\tau_n}^{\varphi_1}/Z_{\tau_n}^{\varphi_2}$ converges to the same limit, as $n\rightarrow\infty$.

	The process $Z_t^{\varphi_2}$ simply counts the number of nodes $v$ in the tree at time $t$, and hence $Z_{\tau_n}^{\varphi_2}=n$.
	Since $\xi_v(t)+1$ is the degree of node $v$ at time $t$, the quotient can be rewritten as
	\begin{equation*}
		\frac{ Z^{\varphi_1}_{\tau_n}}{ Z^{\varphi_2}_{\tau_n}} = \frac1{ n}{ \sum_{v\in \Upsilon_{\tau_n}} h\bigl(\deg(v, \Upsilon_{\tau_n})\bigr)}
		= \frac1{ n}{ \sumk h(k) N_k(n)}= \sumk h(k) P_k(n).
	\end{equation*}
	This is the left side of the corollary, and it remains to identify  the limit given in Proposition~\ref{lemma-conv-phi1-phi2}
	as its right side. The latter is equal to
	\begin{align*}
		\frac{\int_0^\infty \mathrm{e}^{-\malt t}\bbE h(\xi(t)+1)\,dt}{\int_0^\infty \mathrm{e}^{-\malt t}\bbE \bbind_{\{t\ge0\}}\,dt}
		=\frac{\sumk h(k) \malt \int_{0}^\infty \mathrm{e}^{-\malt t} \bbP(\xi(t)+1=k)\,dt}{ \malt \int_{0}^\infty \mathrm{e}^{-\malt t} \,dt}.
	\end{align*}
	The denominator is simply $1$ and the numerator is $\sumk h(k) p_k$ by identifying $p_k$ from Equation (17)
	in \cite{gao2017consistent}.
\end{proof}

\begin{proof}[Proof of Theorem~\ref{thm-consistency-param}]
	Recall the notation $\iota_n(f_\theta)$ for the scaled log likelihood in  \eqref{eqn:iota-nn}.
	We  show below that $\sup_\q\bigl|\iota_n(f_\theta)-\iota(f_\theta)-c_n\bigr|\ra 0$, almost surely, for
	$c_n=n^{-1}\sum_{t=2}^n\log (t-1)$ and
	$$\iota(f_\theta)=\sum_{k=1}^\infty\bigl(\log f_\q(k)\bigr)\, p_{>k}^{\sss (0)}-\log \Bigl(\sum_{k=1}^\infty f_\q(k)p_k^{\sss (0)}\Bigr).$$
	Because $\hat\q_n$ maximizes $\q\mapsto \iota_n(f_\theta)$, it next suffices to show that the limit function
	$\q\mapsto \iota(f_\theta)$ possesses $\q_0$ as a well-separated point of maximum, in the sense of
	\cite{vdvaart2000asymptotic}, Theorem~5.7. Since the function $\q\mapsto \iota(f_\theta)$ is continuous
	on the compact set $\Theta$, this is equivalent to showing that  $\q_0$ is a unique point of maximum.

	Consider the probability distributions $q^\q=(q_k^\q)_{k=1}^\infty$ on $\bbN_+$ defined by
	\begin{equation}
		\label{DefinitionQtheta}
		q_k^\q=\frac1{c(\q)} \frac{f_\q(k)}{f_{\q_0}(k)}p_{>k}^{\sss (0)},
	\end{equation}
	where the norming constant satisfies
	$$c(\q)=\sum_k \frac{f_\q(k)}{f_{\q_0}(k)}p_{>k}^{\sss (0)}= \frac{\sum_kf_\q(k)p_k^{\sss (0)}}{\sum_kf_{\q_0}(k)p_k^{\sss (0)}},$$
	by \eqref{eqn-lemma-limit-degree-dist-equality}. In particular $c(\q_0)=1$ and hence $q^{\q_0}_k=p_{>k}^{\sss (0)}$.
	The Kullback--Leibler divergence of $q^\q$ relative to $q^{\q_0}$  can be seen to be
	equal to $\iota(f_{\q_0})-\iota(f_\theta)$, and is strictly positive unless $q^\q=q^{\q_0}$. The latter is equivalent
	to $f_\q \propto f_{\q_0}$, which is excluded by the identifiability assumption.

	We finish by proving the uniform convergence, where we show that the two components of $\iota_n(f_\q)$
	converge to the two components of $\iota(f_\q)$. By Fubini's theorem the difference of the first components satisfies
	\begin{align*}
		\Bigl|\sum_{k=1}^\infty \log f_\q(k)(P_{>k}(n)-p_{>k}^{\sss (0)})\Bigr|
		 & =\Bigl|\sum_{j=1}^\infty\sum_{k=1}^{j-1}\log f_\q(k)(P_{j}(n)-p_{j}^{\sss (0)})\Bigr|        \\
		 & \le \sum_{j=1}^\infty\Bigl|\sum_{k=1}^{j-1}\log f_\q(k)\Bigr|\, |p_{j}^{\sss (0)}-P_{j}(n)|.
	\end{align*}
	Because $\log f_\q(k)\le C \log k$ and $|u|=u+2u_-$, for every $u\in\bbR$, this is bounded
	above by a multiple of
	$$\sum_{j=1}^\infty j\,\log j(p_{j}^{\sss (0)}-P_{j}(n))+2\sum_{j=1}^\infty j\log j\,(p_{j}^{\sss (0)}-P_{j}(n))_-.$$
	The first series on the left tends to zero almost surely by Corollary~\ref{lemma-hk-pk-conv}  applied with the function $h: j \mapsto j\log j$.
	The terms of the second series are bounded from above by $j\log j \,p_j^{\sss (0)}$ and tend to zero almost surely for every $j$ as $n\to \infty$
	by Corollary~\ref{lemma-hk-pk-conv}  applied with $h: k \mapsto \bbind_{\{k\}}$. Furthermore, the series $\sum_jj\log j \,p_j^{\sss (0)}$  converges.
	Hence, the second series
	tends to zero almost surely by the dominated convergence theorem. The two series give an upper bound
	independent of $\q$ and hence the supremum over $\q$ of the left side of the second last equation tends to zero almost surely.

	For the second component of $\iota_n(f_\q)$, we first note that for some constant $C$ not depending on $\theta$,
	\begin{align*}
		\Bigl|\frac{S_{f_\q}(n)}{n}-\sum_{k=1}^\infty f_\q(k)p_k^{\sss (0)}\Bigr| & =\Bigl|\sum_{k=1}^\infty f_\q(k)(P_k(n)-p_k^{\sss (0)})\Bigr|
		\le C \sum_{k=1}^\infty k\,|p_k^{\sss (0)}-P_k(n)|.
	\end{align*}
	The right side tends to zero almost surely as $n\rightarrow \infty$,
	as in preceding paragraph. It follows that the supremum over $\q$ of the left
	side tends to zero almost surely. The limit $\sum_kf_\q(k)p_k^{\sss (0)}$ is a continuous, positive function
	and hence is bounded away from 0. By the continuous mapping theorem,
	$$\sup_\q \Bigl|\log \frac{S_{f_\q}(n)}{n}-\log\Bigl(\sum_{k=1}^\infty f_\q(k)p_k^{\sss (0)}\Bigr)\Bigr|\asconv 0.$$
	By Lemma~\ref{lem:cesaro-uniform-l1} the C\'esaro's sums $n^{-1}\sum_{t=2}^n\log (S_{f_\q}(t-1)/(t-1))
		=n^{-1}\sum_{t=2}^n\log S_{f_\q}(t-1)-c_n$ have the same limit.
\end{proof}

\begin{lemma}[Uniform Cesàro convergence for processes]
	\label{lem:cesaro-uniform-l1}
	If $Z_\q(1),Z_\q(2),\ldots$ are bounded stochastic processes such that $\sup_\q| Z_\theta(n) - Z_\theta|\rightarrow 0$
	almost surely, for some process $Z_\q$, then the Cesàro means $\bar{Z}_\theta(n)=n^{-1}\sum_{t=1}^nZ_\q(t)$ of $Z_\theta(n)$ satisfy
	$\sup_\q| \bar Z_\theta(n) - Z_\theta|\rightarrow 0$, almost surely.
\end{lemma}

\begin{proof}[Proof of Lemma~\ref{lem:cesaro-uniform-l1}]
	For every $m$ the difference  $|\bar{Z}_\theta(n)-Z_\q|$ is bounded from above by
	$$\frac mn \max_{1\le t\le m}| Z_\theta(t) - Z_\theta|+\frac 1n\sum_{t=m+1}^n| Z_\theta(t) - Z_\theta|.$$
	For every $\e>0$ there exists $m$ so that every term in the second sum is bounded from above by $\e$ and hence this sum divided
	by $n$  is bounded by $\e$.
	The first term tends to zero as $n\ra\infty$, for every fixed $m$. This argument is true also after taking the supremum over
	$\q$ across.
\end{proof}


\begin{proof}[Proof of Lemma~\ref{lemma-two-prob-vk}]
	Since \( \sum_{k=1}^\infty p_k = 1 = \sum_{k=1}^\infty q_k \), we have \( \sum_{k=1}^K (q_k - p_k) = \sum_{k=K+1}^\infty (p_k - q_k)\),
	where the terms of the sums are nonnegative by assumption.
	The strict monotonicity of $v_k$ gives $\sum_{k=1}^K (q_k - p_k) v_k < \sum_{k=K}^\infty (p_k - q_k) v_k$.
	Rearranging the terms gives the desired result.
\end{proof}

\begin{proof}[Proof of Lemma~\ref{lemma-reweighted-prob}]
	Since both sequences sum to $1$, it is impossible that $p_k > q_k$ for every $k \in \bbN_+$.
	If $(w_k)_{k=1}^\infty$ is strictly increasing and $p_j \le q_j$, then
	\begin{equation*}
		\frac{q_{j+1}}{ p_{j+1}} = \frac{w_{j+1}}{ \sum_{k=1}^\infty p_k w_k} > \frac{ w_j }{\sum_{k=1}^\infty p_k w_k} = \frac{q_j}{ p_j } \ge 1.
	\end{equation*}
	By mathematical induction, $p_k < q_k $ for every $k > j$.
	Then $p_k < q_k$ for any $k > K$ and $K+1$ the smallest value $j$ with $p_j<q_j$ (which cannot be $j=1$).

	In the case that $(w_k)_{k=1}^\infty$ is strictly decreasing, the sequence  $w^{-1}_k$ is strictly increasing,
	and we apply the preceding argument with $p_k=q_k w_k^{-1}/\sum_j q_jw_j^{-1}$ and the roles of $p_k$ and $q_k$ swapped.
\end{proof}

\begin{proof}[Proof of Lemma~\ref{lem:monotone-f-neq}]
	A more illustrative view of \eqref{eqn:iota-dot}
	is as follows:
	\begin{equation}
		\begin{aligned}
			\dot{\iota}(f_\theta) & = \sum_{k=1}^\infty \frac{\dot{f}_\theta}{f_\theta}(k) p_{>k}^{\sss (0)}
			- \sum_{k=1}^\infty \frac{ p_{>k}^{0} f_\theta(k) / f_{\theta_0}(k) }{\sum_{j=1}^\infty { p_{>j}^{\sss (0)}  f_\theta(j) / f_{\theta_0}(j)  }} \frac{\dot{f}_\theta}{f_\theta}(k) \\
			                      & = \sum_{k=1}^\infty p_{>k}^{\sss (0)}\frac{\dot{f}_\theta}{f_\theta}(k)  - \sum_{k=1}^\infty q_k^{\sss (0,\theta)} \frac{\dot{f}_\theta}{f_\theta}(k),
		\end{aligned}
		\label{eqn:dot-iota-diff-two-expectation}
	\end{equation}
	where $ q^{\sss (0,\theta)}_k \propto p_{>k}^{\sss (0)} f_\theta(k)/f_{\theta_0}(k)$ is the probability distribution generated by
	reweighting $(p_{>k}^{\sss (0)})_{k=1}^\infty$ with $(f_\theta/f_{\theta_0}(k))_{k=1}^\infty$.

	Fix any $\theta \in \Theta'$ and  assume that $\theta$ renders  $f_\theta/f_{\theta_0}$
	decreasing. By Lemma~\ref{lemma-reweighted-prob} applied with weights $w_k=f_\theta/f_{\theta_0}(k)$,
	there exists $K$ such that $p_{>k}^{\sss (0)} \le q^{\sss (0,\theta)}_k $ for $k \le K$ and $p_{>k}^{\sss (0)} > q^{\sss (0,\theta)}_k$ for $k> K$.
	Then  Lemma~\ref{lemma-two-prob-vk} with $v_k = \dot{f}_\theta/f_\theta(k)$ (understood component-wise)
	and the probability distributions $p_{>k}^{\sss (0)}$ and $q_k^{\sss (0,\theta)}$, shows that
	$\dot \iota(f_\theta) < 0$. If $f_\theta/f_{\theta_0} (k)$ is  increasing, then the same argument applies,
	but we find that $\dot \iota(f_\theta) > 0$.
\end{proof}

For reference, we state the martingale central limit theorem---a version of Theorem 3.2 of \cite{hall2014martingale}.

\begin{proposition}
	\label{prop:martingale-clt-simple}
	Suppose that $X_t$ is a martingale difference series relative to the filtration $\calF_t$.
	If as $n \rightarrow \infty$, $n^{-1} \sum_{t=1}^n \bbE[X_t^2|\calF_{t-1}] \xrightarrow{P} v$ for a positive constant $v$ and $n^{-1}\sum_{t=1}^n \bbE[ X^2_t \bbind_{\{|X_t| > \varepsilon\sqrt{n} \}}| \calF_{t-1}] \xrightarrow{P} 0$ for every $\varepsilon>0$, then $\sqrt{n} \bar{X}_n \rightsquigarrow N(0, v)$.
\end{proposition}

\begin{proof}[Proof of Theorem~\ref{thm-param-clt}]
	By Theorem~\ref{thm-consistency-param} the  maximum likelihood estimator $\hat{\theta}_n$ is consistent
	if $\q_0$ is identifiable. The $\hat\q_n$ will eventually be interior to the parameter set if $\q_0$ is interior,
	and hence satisfy the system of likelihood equations $\dot{\iota}_n(f_{\hat\theta_n})=0$. Thus, the second assertion of the theorem
	follows from the first.

	By a Taylor expansion of the $i$th of the likelihood equations, it can be expanded as
	$0=\dot{\iota}_n(f_{\theta_0})_i+\ddot{\iota}_n(f_{\theta'_{n,i}})_i ( \hat \theta_n - \theta_0)$, for
	$\theta'_{n,i}$ on the line segment between $\hat\theta_n$ and $\theta_0$ and $\ddot{\iota}_n(f_{\theta})_i$ the
	$i$th row of the second derivative matrix $\ddot{\iota}_n(f_{\theta})$. Thus,
	$\ddot{\iota}_n(f_{\theta'_n}) ( \hat \theta_n - \theta_0)= -\dot{\iota}_n(f_{\theta_0})$, where
	$\ddot{\iota}_n(f_{\theta'_n})$ is understood to be the $(d\times d)$ matrix with $i$th row
	$\ddot{\iota}_n(f_{\theta'_{n,i}})_i$, even though the vector $\theta'_{n,i}$ may be different for different $i$.
	The proof can be concluded by showing that
	$\sqrt n \,\dot{\iota}_n(f_{\theta_0})\weak N(0, V_0)$ and that
	$\ddot{\iota}_n(f_{\theta'_n})\rightarrow -V_0$ in probability.

	As noted in Section~\ref{sec:mle}, the sequence $n\dot{\iota}_n(f_{\theta_0})$ is a (vector-valued) martingale. The
	asymptotic normality can be obtained from the martingale central limit theorem (see
	Proposition~\ref{prop:martingale-clt-simple}).
	Because $\bbP_\q(D_t=k\given\F_{t-1})=f_{\q}(k)N_k(t-1)/S_{f_\q}(t-1)$,
	the martingale differences $(\dot f_{\q_0}/f_{\q_0})(D_t)- \bbE_{\q_0}\bigl((\dot f_{\q_0}/f_{\q_0})(D_t)\given \F_{t-1}\bigr)$
	possess conditional covariances
	\begin{align*}\Sigma_t:= & \sumk \Bigl( \frac{\dot f_{\theta_0}}{f_{\theta_0}}\Bigr)\Bigl( \frac{\dot f_{\theta_0}}{f_{\theta_0}}\Bigr)^T(k)
              \frac{f_{\theta_0}(k) P_k(t-1)}{ S_{f_{\q_0}}(t-1)/(t-1)}                                                                                   \\
                         & \qquad- \Bigl( \sumk \frac{\dot f_{\theta_0}}{ f_{\theta_0}}(k)\frac{f_{\theta_0}(k) P_k(t-1)}{ S_{f_{\q_0}}(t-1)/(t-1)}\Bigr)
              \Bigl( \sumk \frac{\dot f_{\theta_0}}{ f_{\theta_0}}(k)\frac{f_{\theta_0}(k) P_k(t-1)}{ S_{f_{\q_0}}(t-1)/(t-1)}\Bigr)^T.
	\end{align*}
	As seen in the proof of Theorem~\ref{thm-consistency-param},  the sequence
	$S_{f_{\q_0}}(t)/t$ tends almost surely to $\sum_j f_{\q_0}(j)p_j^{\sss (0)}$, as $t\rightarrow\infty$.
	Corollary~\ref{lemma-hk-pk-conv} applied with $h$ equal to the entries of the matrix
	$\dot f_{\q_0}\dot f_{\q_0}^T/f_{\q_0}$ or the vector $\dot f_{\q_0}$ shows that the preceding display tends almost surely to
	$$\sumk \frac{\dot f_{\theta_0}\dot f_{\q_0}^T}{f_{\theta_0}}(k)\frac {p_k^{\sss (0)}}{\sum_j f_{\q_0}(j)p_j^{\sss (0)}}
		- \Bigl( \sumk \dot f_{\theta_0}(k)\frac {p_k^{\sss (0)}}{\sum_j f_{\q_0}(j)p_j^{\sss (0)}}\Bigr)
		\Bigl( \sumk \dot f_{\theta_0}(k)\frac {p_k^{\sss (0)}}{\sum_j f_{\q_0}(j)p_j^{\sss (0)}}\Bigr)^T.$$
	In view of equation \eqref{eqn-lemma-limit-degree-dist-equality}, this is equal to the matrix $V_0$.
	The averages $n^{-1}\sum_{t=2}^n\Sigma_t$ of the conditional covariances tend to the same limit, by Lemma~\ref{lem:cesaro-uniform-l1}.

	Because $D_t\le t$, we bound, using \eqref{eqn-partial-quotient},
	$$\Bigl\| \frac{\dot f_{\q_0}}{f_{\q_0}}(D_t)\Bigr\|\le C \log^\gamma t\le C\log ^\gamma n,\qquad t\le n.$$
	As this is smaller than $\epsilon\sqrt n$, eventually for every $\epsilon>0$,
	the conditional Lindeberg condition  is trivially satisfied.
	We conclude that $\sqrt n \,\dot{\iota}_n(f_{\theta_0})\weak N(0, V_0)$, by the martingale central limit theorem, for instance,
	Proposition~\ref{prop:martingale-clt-simple}.

	The Hessian matrix $\ddot{\iota}_n (\theta)$ takes the form
	\begin{align}
		\nonumber
		\ddot{\iota}_n (f_\theta)
		 & = \sumk \Bigl(\frac{\ddot f_\theta}{f_\theta} - \frac{\dot f_\theta \dot f_\theta^T}{ f_\theta^2}\Bigr) (k) P_{>k}(n)
		- \frac{1}{n} \sum_{t=2}^n \Bigl(\frac{S_{\ddot f_\theta}}{S_{f_\theta}}- \frac{S_{\dot f_\theta}S_{\dot f_\theta^T}}{S_{\dot f_\theta}^2}\Bigr)(t-1) \\
		 & = \sum_{j=1}^\infty\sum_{k=1}^{j-1}\Bigl(\frac{\ddot f_\theta}{f_\theta} - \frac{\dot f_\theta \dot f_\theta^T}{ f_\theta^2}\Bigr) (k) P_j(n)
		- \frac{1}{n} \sum_{t=2}^n \Bigl(\frac{S_{\ddot f_\theta}}{S_{f_\theta}}- \frac{S_{\dot f_\theta}S_{\dot f_\theta^T}}{S_{\dot f_\theta}^2}\Bigr)(t-1).
		\label{EqHessian}
	\end{align}
	Using Corollary~\ref{lemma-hk-pk-conv}, the first term in the second line can be shown
	to converge as $n\rightarrow \infty$ to the expression obtained by replacing $P_j(n)$ by $p_j^{\sss (0)}$,
	or equivalently replacing $P_{>k}(n)$ by $p_{>k}^{\sss (0)}$ in the first line.
	For the second term we first note that $S_f(t)/t$ tends almost surely to $\sum_j f(j)p_j^{\sss (0)}$, as $t\rightarrow\infty$,
	for $f$ equal to $\ddot f_{\q}$, $\dot f_\q$ or $f_\q$. By the continuous mapping theorem
	the terms of the sum converge to the corresponding limit.
	The second term then converges almost surely to the same limit, in view of Lemma~\ref{lem:cesaro-uniform-l1}, still uniformly in $\q$.
	By arguments similar to those
	in the proof of Theorem~\ref{thm-consistency-param}, the convergences of both terms can be seen to be uniform in $\q$.

	Finally, the continuity of the limit and consistency of $\hat\q_n'$ for $\q_0$ give that the $\ddot{\iota}_n (f_{\theta_n'}) $
	tends to the limit evaluated at $\q_0$. This can be seen to be equal to the matrix $-V_0$ with the help
	of \eqref{eqn-lemma-limit-degree-dist-equality}, where the two terms involving $\ddot f_\q$ cancel each other.
\end{proof}

\begin{proof}[Proof of Theorem~\ref{cor:affine-test}]
	Because the true parameter is on the boundary of the parameter set, the \mle may not solve the likelihood equations. Instead,
	we use its characterization as the maximizer of the log likelihood. The log likelihood evaluated at the parameter
	$\q_0+h/\sqrt n$ satisfies
	$$\ell_n\bigl(f_{\q_0+h/\sqrt n}\bigr)-\ell_n(f_{\q_0})=h^T\sqrt n\,\dot\iota_n(f_{\q_0})+\thalf h^T\ddot\iota_n(f_{\q_n'(h)})h,$$
	where $\q_n'(h)$ is on the line segment between $\q_0$ and $\q_0+h/\sqrt n$.
	The rescaled \mle $\hat h_n=\sqrt n (\hat\q_n-\q_0)$ maximizes this process over the set $H_n$ of all
	$h=\sqrt n(\q-\q_0)$ such that $\q=(\a,\b)$ belongs to the parameter set $\{(a,\b): \a\in (-1+\e,\e^{-1}), \b\le 1\}$.
	Since $\hat\q_n$ is consistent for $\q_0$, by Theorem~\ref{thm-consistency-param}, the set $H_n$ can be further reduced to
	a set such that $\|h\|<\sqrt n\,\d_n$, for some $\d_n\rightarrow 0$. By the arguments in the proof
	of Theorem~\ref{thm-param-clt}, we have $\sup_\q |\ddot \iota_n(f_\q)-\ddot \iota(f_\q)|\ra 0$, in probability.
	Combined with the continuity of $\q\mapsto \ddot \iota(f_\q)$, we see that
	$\sup_{h\in H_n}|\ddot\iota_n(f_{\q_n'(h)})-\ddot\iota(f_{\q_0})|\ra 0$, almost surely. From the
	nonsingularity and negative definiteness of $-V_0=\ddot\iota(f_{\q_0})$, we then see that
	$h^T\ddot\iota_n(f_{\q_n'(h)})h<-c\|h\|^2$, for every $h\in H_n$ and some $c>0$, with probability tending to one.
	Using that  $0\in H_n$ and $|h^T\sqrt n\,\dot\iota_n(f_{\q_0})|\le \|h\| O_P(1)$, we conclude that
	$\hat h_n=O_P(1)$. Next the argmax continuous mapping theorem (e.g.\ Corollary~5.58 and Lemma~7.13 in
	\cite{vdvaart2000asymptotic}) shows that
	$$\hat h_n=\argmax_{h\in H_n}\bigl(h^T\sqrt n\,\dot\iota_n(f_{\q_0})+\thalf h^T\ddot\iota_n(f_{\q_n'(h)})h\bigr)
		\weak \argmax_{h\in H}(h^TZ_0-\thalf h^TV_0h),$$
	for $H=\{(h_1,h_2): h_2\le 0\}$ the limit of the sequence of sets $H_n$ and $Z_0\sim N(0,V_0)$ the limit in distribution of
	the sequence $\sqrt n\,\dot\iota_n(f_{\q_0})$. The right side has the claimed distribution, by Lemma~\ref{LemmaArgmaxDist} (where we set $V = V_0$).
\end{proof}

\begin{lemma}
	\label{LemmaArgmaxDist}
	If $\hat h=\argmax_{h: a^Th\le 0}(2h^TZ-h^TV_0h)$ for $a\in \RR^d$ and a random variable $Z\sim N_d(0,V)$
	for positive semi-definite matrices $V$ and $V_0$, then $a^T\hat h\sim W\bbind_{W\le 0}$ for $W\sim N(0,a^TV_0^{-1} V V_0^{-1} a)$.
\end{lemma}

\begin{proof}[Proof of Lemma~\ref{LemmaArgmaxDist}]
	Since $2h^TZ-h^TV_0h=-\|V_0^{1/2}h-V_0^{-1/2}Z\|^2+Z^TV_0^{-1}Z$, the variable $\hat h$ can be
	seen to be equal to $V_0^{-1/2}\Pi_G(V_0^{-1/2} Z)$, for $\Pi_G$ the projection onto the half space
	$G:=V_0^{1/2}\{h: a^Th\le 0\}= \{g: b^Tg\le 0\}$, for $b=V_0^{-1/2}a$. Therefore, $a^T\hat h=b^T\Pi_G(\tilde W) = \bbind_{b^T \tilde W \le 0} b^T\tilde W$,
	for $\tilde W:= V_0^{-1/2} Z\sim N_d(0,V_0^{-1/2}V V_0^{-1/2})$.
	The proof is complete by setting $W : = b^T \tilde W = a^T V_0^{-1} Z$.
\end{proof}

\begin{proof}[Proof of Theorem~\ref{TheoremANQMLE}]
	Since $\sum_{k=1}^\infty a_kP_{>k}(n)=\sum_{j=1}^\infty\sum_{k=1}^{j-1}a_kP_j(n)$,
	we can write  $\dot{\tilde{\iota}}_n(f_{\theta_0})=\phi\bigl(S_n)$, for $S_n= \sum_{k=1}^\infty F_{\q_0}(k)P_k(n)$,
	and $\phi: \RR^d\times\RR^d\times \RR\to\RR^d$ given by $\phi(s_1,s_2,s_3)= s_1-\frac1{s_3}s_2$.
	Similarly, it follows from \eqref{eqn-lemma-limit-degree-dist-equality} that
	$\phi(S^{\sss (0)})=0$, for $S^{\sss (0)}=\sum_{k=1}^\infty F_{\q_0}(k)p_k^{\sss (0)}$.
	The assumption on $S_n$ and the delta-method then give that
	the sequence $\sqrt n \dot{\tilde{\iota}}_n(f_{\theta_0})-\sqrt n \bigl(\phi(S_n)-\phi(S^{\sss (0)})\bigr)$
	tends in $\RR^d$ in  distribution to a normal
	distribution with mean zero and covariance matrix $V=(I,-I/S^{\sss (0)}_3, S^{\sss (0)}_2/(S^{\sss (0)}_3)^2)
		W_0(I,-I/S^{\sss (0)}_3, S^{\sss (0)}_2/(S^{\sss (0)}_3)^2)^T$.

	The difference between the pseudo score function $\dot{\tilde \iota}(f_\q)$ and $\dot\iota(f_\q)$ is only
	in their second terms, which are $S_{\dot f_\q}/S_{f_\q}(n)$ for the first and the C\'esaro averages of these for the
	second. The Hessian matrices $\ddot{\tilde \iota}(f_\q)$ and $\ddot\iota(f_\q)$ also only differ by the
	derivatives of the second terms. In the proof of Theorem~\ref{thm-param-clt} it was noted that the latter are
	again C\'esaro averages (see \eqref{EqHessian}) and they were shown to converge by showing convergence of the
	individual terms. Thus, this proof applies also in the present situation, and the full proof can be
	finished as the proof of Theorem~\ref{thm-param-clt}.
\end{proof}

\begin{proof}[Proof of Corollary~\ref{thm:pmle-asymp-normal-trunc}]
	In view of Theorem~\ref{TheoremANQMLE} it suffices to verify the asymptotic normality of the
	sequence $\sum_{k=1}^\infty F_{\q_0}(k)\sqrt n(P_k(n)-p_k^{\sss (0)})$, for $F_{\q_0}$ as given.
	Because $f_{\q}(k)$ is constant in $k\ge K$, so is $F_{\q_0}(k)$ and hence
	$\sum_k F_{\q_0}(k)P_k(n)= \sum_{k\le K}F_{\q_0}(k)P_k(n)+F_{\q_0}(K)P_{>K}(n)$, and the same
	for $p_k^{\sss (0)}$ instead of $P_k(n)$. The desired convergence therefore follows from
	Proposition~\ref{prop:asymp-normality-trunc} and the continuous mapping theorem.
\end{proof}

\newpage
\appendix
\section{Branching Processes and Rooted Ordered Trees}
\label{sec:rooted-tree}
A \textit{rooted ordered tree} is a tree in which one node is designated as the \textit{root}
and the other nodes can be oriented in parent-child relations in reference to their distance to this root node.
In a dynamic setup the root is the initial ancestor who is responsible
for giving births directly or indirectly to every other node.  In the Ulam--Harris labelling notation
for  branching processes, the root is denoted by $\emptyset$ and every
other node has the form $(i_1,\dots,i_l)$, for positive natural numbers $i_j\in\bbN_{+}$ with $l\in\bbN_{+}$.
The node $x = (i)$ is the $i$-th child of the root, and more generally
the node $x = (i_1,\dots,i_k)$ is the $i_k$-th child of $\tilde{x} = (i_1,\dots,i_{k-1})$.
By induction, the set of all possible individuals is
\begin{equation*}
	\scrI= \lrbkt{\emptyset} \cup \left( \bigcup_{k=1}^\infty \bbN_{+}^k\right).
\end{equation*}
The root is the zero-th generation and the $k$-th generation consists of all $x \in \bbN_+^k$.
For $x = (x_1, \ldots, x_k)$ and $y = (y_1, \dots, y_l)$ the notation $xy$ is shorthand for the concatenation
$(x_1, \dots, x_k, y_1, \dots, y_l)$, and, in particular, $xl=(x_1, \dots, x_k, l)$.
The labelling of the nodes contains all parental information and a rooted ordered tree is defined to be
a subset $G \subset \scrI$ such that 
if $x= (x_1, \ldots,x_k)\in G$, then both $(x_1,\ldots, x_{k-1})\in G$ and $(x_1,\ldots, x_k-1)\in G$ in case $x_k\ge 2$.
A corresponding graphical representation is obtained by drawing a node for every $x\in G$ and
connecting two nodes by an edge if they are in a parent-child relationship.
The \textit{degree} of node $x \in G$ is
\begin{equation}
	\deg (x,G) =|\lrbkt{l \in \bbN_{+}\mid xl\in G }|+1,
	\label{eqn-ee-deg-def}
\end{equation}
where the extra one is for the parent of the node.

To set up a stochastic branching process, each individual $x \in \scrI$ is associated with a
stochastic variable $\lambda_x$ and two stochastic processes $\xi_x$ and $\varphi_x$, called the
\textit{life length},  the \textit{reproduction process} and the \textit{characteristic} of $x$. The triples
$(\lambda_x,\xi_x,\varphi_x)$ are taken \iid across the nodes $x\in\scrI$. Formally they
may be defined as copies on a product probability space
\begin{equation*}
	(\Omega, \calB, P) =  \prod_{x \in \scrI} (\Omega_x, \calB_x, P_x),
\end{equation*}
where  each $(\Omega_x, \calB_x, P_x) $ is a copy of a probability space $(\Omega_0, \calB_0, P_0)$.
For a given measurable map $(\lambda,\xi,\varphi)$ defined on $(\Omega_0, \calB_0, P_0)$, we then define
$(\lambda_x,\xi_x,\varphi_x)(\omega)= (\lambda,\xi,\varphi)(\omega_x)$ if $\omega=(\omega_x)_{x\in\scrI}\in \Omega$.
The life length $\lambda$ is a nonnegative
random variable, which in our case we take identically $\infty$ (no node will die). The reproduction
process $\xi=(\xi(t): t\ge 0)$ will be a counting process starting with $\xi(0)=0$ and increasing by steps
of size $1$ at random times. We identify $\xi$ with a random $\bbN_+$-valued measure through $ \xi([0,t])=\xi(t) $
and denote by $\mu : t \mapsto \bbE[\xi(t)]$ its mean (or intensity) measure,
which is often called the \textit{reproduction function} in this context. The characteristic
will also be a stochastic process $\varphi=(\varphi(t): t\ge 0)$, where we may set $\varphi(t)=0$ for $t<0$.

The  random point process $\xi_x$  models the birth times of the children of individual $x$ relative to
the birth time $\sigma_x$ of $x$. The latter birth times are formally defined recursively
by setting the birth time of the root $\emptyset$ at $t = 0$ (hence $\sigma_\emptyset = 0$), and next
the birth time $\sigma_y$ of $y$ in calendar time by
\[
	\sigma_y = \sigma_x + \inf\{t\ge 0\colon \xi_x(t)\ge l \},\qquad \text{ if }y = xl.
\]
The \textit{calendar time} is the evolution time of the branching process, as opposed to the
local time scales of the processes $\xi_x$ and $\varphi_x$, of which the local zero time is interpreted to be
$\sigma_x$ in calendar time. The variable $\varphi_x(t)$ is interpreted as the  characteristic of individual $x$
when $x$ has age $t$. Calendar time is also different from the discrete time steps used to describe the
evolution of a \pa network.

For a given characteristic we define the process
\begin{equation*}
	Z_t^{\varphi} = \sum_{x \in \scrI\colon \sigma_x\le t} \varphi_x(t- \sigma_x).
\end{equation*}
The variable $t-\sigma_x$ is the time since birth of individual $x$ and hence
$\varphi_x(t- \sigma_x)$ can be interpreted as the characteristic of individual $x$ at calendar time $t$.
The variable $Z_t^\varphi$ is the sum of all such characteristics over the individuals that are alive at time $t$.

A branching process is \textit{supercritical} and \textit{Malthusian} if its reproduction function $\mu$
does not concentrate on any lattice $\lrbkt{0, h , 2h,\dots}$, for some $h >0$, and there
exists a number $\malt>0 $ such that
\begin{equation}
	\int_0^{\infty} {\mathrm e}^{-\malt t} \mu(dt) = 1.
	\label{eqn:network-ee-malthusian}
\end{equation}
We shall also assume the integrability assumption
\begin{equation}
	\int_0^\infty t^2 \mathrm{e}^{-\malt t}\, \mu(dt) < \infty.
	\label{eqn:network-ee-lambda-finite}
\end{equation}
Existence of a solution to equation \eqref{eqn:network-ee-malthusian} is called the \textit{Malthusian} assumption,
and $\malt$ is called the Malthusian parameter.

The following proposition can be obtained by combining Theorem 3.1, Corollary 3.4
and Theorem 6.3 of \cite{nerman1981convergence}.
Define $_\lambda \xi(t) = \int_0^t e^{-\lambda u}\, \xi(du)$.

\begin{proposition}
	\label{prop-conv-prob}
	Assume that the reproduction function $\mu(t) = \bbE[\xi(t)]$ satisfies conditions \eqref{eqn:network-ee-malthusian}
	and \eqref{eqn:network-ee-lambda-finite} and does not concentrate on any lattice.  Assume
	that $t\mapsto \bbE[\varphi(t)]$ is continuous almost everywhere with respect to the Lebesgue measure
	and the following conditions hold:
	\begin{gather}
		\label{eqn-prob-conv-sup-condition}
		\sum_{k=0}^\infty \sup_{ k\le t \le k+1} \big( e^{-\malt t} \bbE[\varphi(t)] \big) < \infty,\\
		\bbE[\sup_{s \le t} \varphi(s) ] < \infty \quad \text{for all } t < \infty.  \label{eqn-prob-conv-exp-condition}
	\end{gather}
	Then  there exists a random variable $Y_\infty$ depending only on the reproduction process $\xi(t)$ such that,   as $ t \rightarrow \infty$,
	\begin{equation}
		e^{-\malt t} Z_t^{\varphi} \xrightarrow{P} Y_{\infty} m_{\infty}^\varphi,
		\label{eqn-conv-in-prob}
	\end{equation}
	where $m_\infty^\varphi$ is defined as
	\begin{equation*}
		m_\infty^\varphi = \frac{ \int_0^\infty e^{- \malt t} \bbE [ \varphi(t) ] \,dt}{ \int_{0}^\infty t e^{-\malt t}\, d\mu(t)}.
	\end{equation*}
	The convergence in \eqref{eqn-conv-in-prob} also holds in the $L_1$ sense if
	\begin{equation}
		\bbE[_\malt \xi(\infty) \log^{+}{}_\malt \xi(\infty)] <\infty.
		\label{eqn-xi-logp}
	\end{equation}
	Suppose that the reproduction process  $\xi$ satisfies \eqref{eqn-xi-logp}, and both $\varphi_1$ and $\varphi_2$ satisfy the conditions
	\eqref{eqn-prob-conv-sup-condition} and \eqref{eqn-prob-conv-exp-condition}.
	Define $T_t $ as total number of births up to and including time $t$.
	Then, on the event $\lrbkt{T_t \rightarrow \infty}$, as $t \rightarrow \infty$.
	\begin{equation}
		\frac{Z_t^{\varphi_1}}{Z_t^{\varphi_2}} \xrightarrow{P} \frac{m_\infty^{\varphi_1}}{ m_\infty^{\varphi_2}}
		= \frac{ \int_0^\infty e^{- \malt t} \bbE [ \varphi_1(t) ] \, dt }{\int_0^\infty e^{- \malt t} \bbE [ \varphi_2(t) ]\, dt} .
		\label{eqn-conv-quotient}
	\end{equation}
	If $\varphi_1$ and $\varphi_2$ have càdlàg paths and there exists a $\lambda < \malt$ such that
	\begin{align}
		\bbE[_\lambda\xi(\infty) ]                          & < \infty, \label{eqn:V-beta} \\
		\bbE \bigl[\sup_t e^{-\lambda t} \varphi_i(t)\bigr] & <\infty,\qquad i=1,2,
		\label{eqn:beta-xi}
	\end{align}
	then, on $\lrbkt{T_t \rightarrow \infty}$, the convergence in \eqref{eqn-conv-quotient} is also in the almost sure sense.
\end{proposition}

\section{Asymptotic normality of Empirical Degrees}
\label{sec:janson-truncated}
In this section we derive the asymptotic normality of the empirical degrees in some cases of the preferential attachment model,
using an urn process studied by \cite{janson2004functional}.

The urn process consists of vectors $X_n = (X_{n,1}, \dots, X_{n,q})^T$ in $[0,\infty)^q$, of which the $i$-th coordinate
represents the quantity in the $i$-th urn at time $n$.
Given are, for each $i \in [q]$, an `activity' $a_i \ge 0$  and a vector $\xi_{i} = (\xi_{i,1}, \dots, \xi_{i,q})^T\in\RR^q$ with
$\xi_{i,j}\ge 0$ for $j \neq i$ and $\xi_{ii} \ge -1$.
The process $(X_n)_{n=0}^\infty$ evolves as a Markov process, with transitions determined by: given $X_{n-1}$,
\begin{enumerate}
	\item pick an urn $i\in [q]$ with probability $a_i X_{n-1,i}/\sum_{j=1}^q a_j X_{n-1,j}$;
	\item set $X_{n} := X_{n-1} + \xi_i$.
\end{enumerate}
It is assumed that
$\sum_{j=1}^q\xi_{i,j}\ge 0$, for every $i\in[q]$, with strict inequality for some $i$, so that the
total content $\sum_{i=1}^q X_{n,i}$ of the urns is nondecreasing.
(Actually, \cite{janson2004functional} allows the vectors $\xi_i$ to be random, but
deterministic vectors suffice in our situation, and allow a simpler statement of the main
result.)

Define the `transfer' matrix $A\in \bbR^{q \times q}$ by
\begin{equation}
	A_{ij} = a_j \xi_{j,i}.
	\label{eqn:def-A}
\end{equation}
We assume that $A$ is irreducible.
By an application of the Perron--Frobenius theorem (to the nonnegative matrix  $A+\a I_q$, for sufficiently large $\a$),
it can be seen that the eigenvalue of $A$ with the largest real value is real, and
the eigenvalues can be ranked by their real parts as $\lambda_1 > \Re \lambda_2 \ge \Re \lambda_3 \ge \cdots$.
Furthermore, $\l_1>0$, has multiplicity one, and the associated eigenvector $v_1$ has all positive coordinates.
We normalize $v_1$ such that $a^T v_1 = 1$, where $a=(a_1,\ldots, a_q)^T$.

In this setup, conditions (A1)-(A6) in \cite{janson2004functional} are satisfied (see his Lemma~2.1).
A crucial further assumption in the following proposition, which restates Theorems~3.21-3.22 of \cite{janson2004functional},
is that $\Re \lambda_2<\l_1/2$.

Define the following quantities:
\begin{gather}
	B := \sum_{i=1}^q v_{1i} a_i \xi_i \xi_i^T, \label{eqn:janson-B} \\
	\varphi(s, A) := \sum_{n=1}^\infty \frac{s^n}{ n!} A^{n-1} = \int_{0}^s e^{tA}\, dt, \\
	\psi(s, A) := e^{sA} - \lambda_1 v_1 a^T \varphi(s,A). \label{eqn:janson-psi}
\end{gather}

\begin{proposition}
	Under the preceding conditions, as $n \rightarrow \infty$,
	\[n^{-1} X_n \xrightarrow{\text{a.s.}} \lambda_1 v_1.\]
	If, moreover,  $\Re \lambda_2 < \lambda_1/2$, then, as $n \rightarrow \infty$,
	\[ n^{1/2} (n^{-1} X_n -  \lambda_1 v_1) \xrightarrow{\text{d}} N(0, \Sigma),\]
	where the covariance matrix $\Sigma$ is defined as (with quantities defined in \eqref{eqn:janson-B}--\eqref{eqn:janson-psi})
	\begin{equation}
		\Sigma = \int_{0}^\infty \psi(s,A) B \psi(s,A)^T e^{-\lambda_1 s} \lambda_1 ds - \lambda_1^2 v_1 v_1^T.
		\label{eqn:janson-Sigma}
	\end{equation}
	\label{prop:janson-conv}
\end{proposition}

The preferential attachment model would be naturally described using an infinite number of urns,
with the content $X_{n,i}$ of the $i$-th urn corresponding to the number of nodes of degree $i$ at
time $n$. The activities $a_i$ can then be set equal to the preferences $f(i)$ and the vectors $\xi_i$ defined by
their coordinates
\begin{equation}
	\label{EqDefxi}
	\xi_{i,j} = - \bbind_{\{ j = i\}} + \bbind_{\{ j = i +1\}} + \bbind_{\{j=1\}}.
\end{equation}
The last term $\bbind_{\{j=1\}}$ corresponds to the new node, which has degree 1, and is counted
in the first urn, while the first two terms on the right describe the decrease and increase by 1 of the
numbers of nodes of degrees $j$ and $j+1$,  respectively, if the new node is attached to an existing node
of degree $j$.  There would then be infinitely many vectors ($i\in \bbN_+$) each with infinitely many coordinates
($j\in\bbN_+$), but, unfortunately, Proposition~\ref{prop:janson-conv} allows a fixed, finite number of urns only. In
the following, we consider two examples of \pa models
in which the infinite process can be reduced to a finite number: the case of an affine
\pa function, and the case that the \pa function is constant from a fixed degree onward.

In the affine case with \pa function $f(k)=k+\a$, we study the degree distribution up to some given degree $\k$ by gathering
all nodes of degree strictly bigger than $\k$ in a single urn, labelled $\k+1=:q$. To accommodate that the latter nodes
have different preferences, we define the vectors $\xi_i\in\bbR^{\k+1}$ for $i=1,\ldots, \k-1$ as in \eqref{EqDefxi} with coordinates
restricted to $j\in [\k+1]$, but redefine
$\xi_\k$ and $\xi_{\k+1}$ by
\begin{align}
	\xi_{\kappa,j}    & = - \bbind_{\{ j = \kappa\}} + \bbind_{\{ j = \kappa +1\}} (\kappa + 1 + \a) + \bbind_{\{j=1\}},\label{Eqxik} \\
	\xi_{\kappa +1,j} & = \bbind_{\{j = \kappa +1 \}}  + \bbind_{\lrbkt{j=1}}.\label{Eqxikplusone}
\end{align}
We combine this with the vector of activities $a_\k= (1+\a, 2 + \a,\dots, \kappa +\a, 1)^T$. Thus, urns
$1,\ldots,\k$ have activities equal to the preferential attachment function, but urn $\k+1$ has activity 1.
With these definitions, for $i=1,\ldots,\k$ the variable $X_{n,i}$ corresponds to the numbers of nodes
of degree $i$, but $X_{n,\k+1}$ is set to correspond to the total preference of all nodes of degree bigger than $\k$.
Indeed, a choice of an urn $i=1,\ldots,\k-1$ follows the scheme described before, with the transition given by \eqref{EqDefxi}.
Second, a choice of urn $\k$ corresponds to choosing a node of degree $\k$; by \eqref{Eqxik}
the count of urn $\k$ is then decreased by one,
and $\k+1+\a$ balls are added to urn $\k+1$, each weighted by activity $a_{\k+1}=1$,
thus giving the correct increase of total preference of the nodes of degree bigger than $\k$. Third, a choice
of urn $\k+1$ corresponds to choosing a node of some degree bigger than $\k$; this node is replaced
by a node of degree one bigger, resulting in an increase by 1 of the total preference of the nodes of degree bigger than $\k$.
Thus,  \eqref{Eqxikplusone} correctly changes the total preferences of the nodes of degree bigger than $\k$.
In both \eqref{Eqxik} and \eqref{Eqxikplusone} the term on
the far right corresponds to the new node of degree 1, counted in $X_{n,1}$.

Asymptotic normality of the empirical degrees $P_k(n)$ in the affine case was first proved in  \cite{mori2002random}
and \cite{resnick2016asymptotic}.
We deduce it here by a simple argument based on the preceding proposition,

\begin{proposition}
	In the \pa model with \pa function $f(k) = k + \alpha$, the centered and rescaled empirical degree distribution
	$\bigl( \sqrt{n} \bigl(P_k(n) - p_k \bigr), k = 1,2,\dots \bigr)$ converges in distribution in $\bbR^{\bbN_+}$ to
	a centered Gaussian process.
	\label{prop:asymp-normality-affine}
\end{proposition}

\begin{proof}[A novel proof]
	We fix arbitrary $\k\in\NN$, and define
	activities $a_i$ and update vectors $\x_i$, for $i\in [\k+1]$, as indicated in \eqref{EqDefxi}--\eqref{Eqxikplusone},
	and initial vector $X_0 = (1, 0, \dots,0)^T$. For every $n$,
	the vector $(X_{n,1},\ldots, X_{n,\k}, X_{n,\k+1})$
	is then identically distributed to the vector $\bigl(N_1(n),\ldots, N_\k(n), \sum_{j>\k}N_j(n)(j+\a)\bigr)$.
	Since (see, e.g., \citet[Example 2.4]{billingsley2013convergence} or \citet[Theorem~1.6.1]{vandervaartwellner})
	weak convergence in $\bbR^{\bbN_+}$
	is the same as convergence of all finite marginals, the first assertion is proved if we can
	prove  convergence of these vectors for every fixed $\k$.
	For this we apply Proposition~\ref{prop:asymp-normality-affine}.

	The transfer matrix $A=A_\kappa$ is given by
	\begin{equation*}
		A_\kappa = \begin{pmatrix} 
			0           & \  2 + \a & \  3 + \a    & \cdots   & \kappa-1+\a & \kappa  + \a                       & 1      \\
			1 + \a \ \  & -2 - \a   & 0            & \cdots   & 0           & 0                                  & 0      \\
			0           & \ 2 + \a  & -3 - \a \ \  & \ \cdots & 0           & 0                                  & 0      \\
			0           & 0         & \  3+ \a     & \cdots   & 0           & 0                                  & 0      \\
			\vdots      & \vdots    & \vdots       & \vdots   & \vdots      & \vdots                             & \vdots \\
			0           & 0         & 0            & \cdots   & -\k+1-\a    & 0                                  & 0      \\
			0           & 0         & 0            & \cdots   & \ \k-1+\a   & -\kappa  - \a                      & 0      \\
			0           & 0         & 0            & \cdots   & 0           & \ (\kappa  + \a)( \kappa+1 + \a)\  & 1
		\end{pmatrix}.
	\end{equation*}
	We can calculate that $\det(\lambda I - A_2) = (\lambda - 2-\a)(\lambda + 1 + \a)(\lambda + 2 + \a) $.
	Furthermore, by subtracting $(\k+\a)$ times the $(\k+1)$-th column from the $\k$-th column, next
	pulling out the factor $\l+\k+\a$ from the $\k$-th column and finally adding $(\k-1+\a)$ times the $\k$-th column to the
	$(\k-1)$-th column, we can see
	that $\det (\lambda I - A_{\kappa} )= (\lambda + \kappa + \a)\det(\lambda I - A_{\kappa-1})$.
	Therefore, by  mathematical induction
	\begin{equation*}
		\det(A_\kappa - \lambda I) = \bigl(\lambda - (2+ \a)\bigr)  \prod_{l=1}^{\kappa} \bigl(\lambda + (l + \a)\bigr).
	\end{equation*}
	We conclude that all eigenvalues are real, and that the only positive eigenvalue is $\lambda_{1,\kappa} = 2 + \a$,
	so that certainly $\Re \l_{2,\k}<\l_{1,\k}/2$. The transfer matrix $A_\kappa$ is irreducible.

	Thus, the conditions of Proposition~\ref{prop:janson-conv} are verified. It suffices to identify the limiting
	mean.

	By (8.6.12) of \cite{hofstadcomplexnet}, the limiting degree distribution for \pa trees with the \pa function $f: k \mapsto  k+\a$
	has coordinates
	\begin{equation*}
		p_k = (2+\a) \frac{\Gamma(k+\a) \Gamma(3+2\a)}{\Gamma(k+3+2\a)\Gamma(1+\a)}.
	\end{equation*}
	This corresponds to the recursion (with $p_1 = (2+\a)/(3+2\a)$)
	\begin{equation}
		p_k = \frac{k-1+\a}{ k+2+2\a} p_{k-1}.
		\label{eqn:pk-recursion}
	\end{equation}
	We now verify that $v_{1,\kappa} := (2+\a)^{-1} \bigl(p_1, p_2, \dots, p_\kappa, \sum_{l=\kappa+1}^\infty p_l(l + \a) \bigr)$
	is an eigenvector associated with the eigenvalue $2 + \a$ with $a_\kappa^T v_{1,\kappa} = 1$, and hence
	obtain that $n^{-1} X_n\ra \l_{1,\k}v_{i1,\k}=\bigl(p_1, p_2, \dots, p_\kappa, \sum_{l=\kappa+1}^\infty p_l(l + \a) \bigr)$,
	almost surely.

	The first coordinate of $A_\kappa (2+\a) v_{1,\kappa}$ is
	$\sum_{l > 1} p_l (l + \a ) = (2+\a) - p_1(1+\a) = (2+\a) p_1$,  by Lemma~\ref{lemma:last-row-eigenv}.
	The second to $\kappa$-th coordinates are equal to $(2+\a)p_k$, for $k=2,\ldots,\k$,
	by the relation \eqref{eqn:pk-recursion}.  The $(\k+1)$-st coordinate
	of $A_\kappa (2+\a) v_{1,\kappa}$ is $\sum_{l > \kappa} p_l (l + \a) + (\kappa + \a)(\kappa+1 + \a)p_k$,
	which coincides with $(2 + \a) \sum_{l > \kappa} p_l (l + \a)$, again by Lemma~\ref{lemma:last-row-eigenv}.
\end{proof}

The covariance function of the limiting Gaussian vector can be obtained from \eqref{eqn:janson-Sigma}
by somewhat tedious calculations, which we omit.
We refer to  (4.28) of \cite[page 18]{resnick2016asymptotic} for its exact form.

By the simple linear relation (with drift)
\( P_{>k}(n) = 1 - \sum_{j=1}^k P_j(n) \), and the continuous mapping theorem, it follows from the
preceding theorem that,  for any  $\kappa\in\bbN_+$,
\begin{equation*}
	\left(  \sqrt{n} (P_{>k}(n) -p_{>k}); k=1,\ldots, \k\right)\weak N(0, R_{\kappa}),
\end{equation*}
The covariance matrix $R_{\kappa}$ in this limit takes a simple form, first pointed out in \cite{mori2002random},
given by
\begin{equation}
	(R_\k)_{ij} = \bbind_{\{i=j \}} p_i (1-p_i) - \bbind_{\{ i \neq j \}} p_i p_j.
	\label{eqn:R}
\end{equation}

\begin{lemma}
	The limiting degree distribution $(p_k)_{k=1}^\infty$ in the
	affine \pa model with the \pa function $f(k) = k + \a$, satisfies, for $ k \in \bbN_+$,
	\begin{equation*}
		\sum_{l > k} p_k (k + \a) = \frac{(k + \a)( k + 1 + \a)}{ 1 + \a} p_k.
	\end{equation*}
	\label{lemma:last-row-eigenv}
\end{lemma}

\begin{proof}
	For  $k = 1$ the left side of the lemma is equal to
	\begin{align*}
		\sum_{l=1}^\infty p_l ( l + \a) - p_1(1+\a)
		 & = 2 + \a - \frac{(2+ \a)(1+\a)}{(3+2\a)}  = \frac{(1+\a)(2+\a)}{1 + \a}p_1.
	\end{align*}
	This proves the claim for $k=1$. We proceed by mathematical induction.
	If the statement is true for any integer up to $k-1$, then the left side of the lemma is equal to
	\begin{align*}
		\sum_{l > k- 1} p_l(l+\a)  - p_k (k+\a)
		 & = \frac{(k-1 +\a) (k+\a)}{1+\a} p_{k} \frac{k + 2+2\a}{k - 1 + \a} - p_k (k+\a),
	\end{align*}
	by the induction hypothesis and the relation \eqref{eqn:pk-recursion} between $p_k$ and $p_{k-1}$.
	The right side can be reduced to the right side of the lemma.
\end{proof}

As a second application of Proposition~\ref{prop:janson-conv}, we obtain the asymptotic normality
of the empirical degrees in \pa models with \pa function that is constant eventually.
From our numerical experiments, we infer that the key eigenvalue condition $\Re \lambda_2 (A) < \lambda_1(A) /2$
is generally satisfied when the \pa function  is sublinear,
but we do not know general conditions for this.

To apply Proposition~\ref{prop:janson-conv} to a \pa model  with eventually constant \pa function,
we simply gather all nodes of degrees
higher than a cut-off $\k$ after which the \pa function is constant (not necessarily the smallest such value)
in a single urn, the $(\k+1$)-th one.
Since the preferences of the corresponding nodes are equal, it is not
necessary to keep track of the different degrees of the nodes inside this bin when studying the lower degrees. The evolution of
the empirical degrees $\bigl(P_1(n),\ldots, P_\k(n), P_{>\k}(n)\bigr)$  will be the same
as the evolution of the vectors $(X_{n,1},\ldots, X_{n,\k}, X_{n, \k+1})$, if we define the
preferences as $a_i=f(i\wedge \k)=f(i)$, and the transition vectors $\xi_i$ for $i\in [\k]$ by
\eqref{EqDefxi}--\eqref{Eqxikplusone} with $j$ restricted to coordinates $j\in [\k+1]$ and
$\xi_{\k+1}=(1,0,0,\ldots, 0)^T$. The last definition corresponds to adding a ball to urn 1 (counting the added node of degree
1) and moving a ball within the $(\k+1)$-st urn (for attaching this node to a node of degree $j>\k$), so not changing
$X_{n,\k+1}=N_{>\k}(n)$.

\begin{proposition}
	In the \pa model with \pa function \( f\) that is  constant on $[\k,\infty)$ such that the matrix
	$A_\k$ defined in \eqref{eqn:A-K} satisfies $\Re \lambda_2(A_\k) < \lambda_1(A_\k)/2$,
	there exist a probability distribution $(p_k)$ such that
	the sequence $\sqrt{n} \bigl(P_1(n) - p_1,\ldots, P_\k(n)-p_\k, P_{>\k}(n)-p_{>\k}\bigr) $ tends
	to a centered normal distribution.
	If the \pa function depends continuously on a parameter $\theta$, pointwise,
	then $p_k$ and the covariance matrix depend continuously on $\q$ as well. If the condition
	on $A_\k$ holds for every sufficiently large $\k$, then the sequence
	$\bigl( \sqrt{n} \bigl(P_k(n) - p_k); k=1,2,\ldots\bigr) $ converges in $\bbR^\infty$.
	\label{prop:asymp-normality-trunc}
\end{proposition}

\begin{proof}
	For the given $\k$ define an urn process as indicated preceding the proposition.
	The matrix $A_\k$ is given by
	\begin{equation}
		\label{eqn:A-K}
		A_\k=\left( \begin{matrix} 0      & f(2)      & f(3)      & f(4)      & \cdots & f(\k)      & f(\k+1) \\
               f(1)   & -f(2)\ \, & 0         & 0         & \cdots & 0          & 0       \\
               0      & f(2)      & -f(3)\ \, & 0         & \cdots & 0          & 0       \\
               0      & 0         & f(3)      & -f(4)\ \, & \cdots & 0          & 0       \\
               0      & 0         & 0         & f(4)      & \cdots & 0          & 0       \\
               \vdots & \vdots    & \vdots    & \vdots    &        & \vdots     & \vdots  \\
               0      & 0         & 0         & 0         & \cdots & -f(\k)\ \, & 0       \\
               0      & 0         & 0         & 0         & \cdots & f(\k)      & 0       \\\end{matrix}\right).
	\end{equation}
	The matrix  $A_\k$ is irreducible, and its eigenvalues satisfy $\Re \lambda_2 < \lambda_1/2$, by assumption.
	Thus, the vector of empirical degrees
	$\bigl(P_1(n),\ldots, P_\k(n), P_{>\k}(n)\bigr)$, suitably centered and scaled,
	is asymptotically normal by Proposition~\ref{prop:janson-conv}. If this is true
	for every sufficiently large $\k\in\bbN_+$, the infinite sequence in the final assertion of the proposition
	converges as well.

	It remains to sort out the continuity of the asymptotic mean and covariance if the \pa
	function depends continuously on a parameter.
	Denote the parameter by $\q$ and set $q=\k+1$.
	By its definition the map $\theta \mapsto A_\k(\theta)$ inherits the continuity from the \pa function.
	Employing the $\min$-$\max$ formula (\citet[Corollary 8.1.31]{horn2012matrix}),
	\[
		\lambda_1(\theta) = \max_{x > 0} \min_{ 1\le i \le q} \frac{1}{x_i} \sum_{j=1}^q \bigl(A_\k(\theta)\bigr)_{ij} x_j,
	\]
	where $x >0$ is understood component-wise.
	By the maximum theorem (e.g.\ \citet[page 229]{ok2011real}), $\theta \mapsto \lambda_1(\theta)$ is continuous.
	The corresponding eigenvector can be obtained as $\adj(C(\q))e_1$, for $\adj(C)$ the adjugate matrix of
	$C=A_\k - \lambda_1 I_q$. This follows, because $C\adj(C)  = (\det C) I_q  = 0$ and $\adj(C)e_1\not=0$. The latter can be seen from
	the fact that the range of $C$, which has dimension $q-1$, is the null space of $\adj(C)$, as $\adj(C) C=0$,  and
	$e_1$ is not in the range of $C$.
	Since the adjugate matrix depends continuously on $A$ and $\l_1$, so does the eigenvector $\adj(C(\q))e_1$,
	and this remains valid after scaling.
	Inspecting the quantities in \eqref{eqn:janson-B}--\eqref{eqn:janson-psi} and \eqref{eqn:janson-Sigma}, we see that
	the asymptotic covariance matrix is continuous.
\end{proof}

\paragraph*{Acknowledgements} We thank Meiyue Shao for his comments regarding the Perron--Frobenius eigenvalue of a nonnegative, irreducible matrix.

\newpage

\bibliographystyle{wangtengyao}
\bibliography{pam.bib, stupidreferee2.bib}

\end{document}